\documentclass[final,hidelinks,onefignum,onetabnum,reqno]{siamart220329}

\usepackage[utf8]{inputenc}
\usepackage[T1]{fontenc}
\usepackage{fullpage}
\usepackage{amsmath}
\usepackage{amssymb}
\usepackage{bm, bbm}
\usepackage{mathtools}
\usepackage{graphicx}
\usepackage{hyperref}
\usepackage{multirow}
\usepackage{diagbox}
\usepackage{cleveref}
\hypersetup{
	colorlinks = true,
	linkcolor = {magenta},
	citecolor= {cyan}
}
\usepackage{xcolor}
\usepackage{thm-restate}
\usepackage{enumerate}
\usepackage{algorithm}
\usepackage{algpseudocode}  
\usepackage{soul}
\sethlcolor{green}
\usepackage[caption=false]{subfig}
\usepackage{anyfontsize, xfp}
\newcommand{\half}{{\frac{1}{2}}}

\newcommand{\RR}{{\mathbb{R}}}
\newcommand{\Ker}{{\operatorname{Ker}}}
\newcommand{\vertiii}[1]{{\left\vert\kern-0.25ex\left\vert\kern-0.25ex\left\vert #1 
		\right\vert\kern-0.25ex\right\vert\kern-0.25ex\right\vert}}
\newcommand{\diag}{{\operatorname{diag}}}
\newcommand{\dist}{{\operatorname{dist}}}

\newcommand{\interior}{\operatorname{int}}
\newcommand{\support}{\operatorname{supp}}
\newcommand{\Span}{{\operatorname{span}}}

\newcommand{\Range}{{\operatorname{Range}}}

\newsiamremark{remark}{Remark}

\makeatletter
{\small 
\xdef\f@size@small{\f@size}
\xdef\f@baselineskip@small{\f@baselineskip}
\normalsize 
\xdef\f@size@normalsize{\f@size}
\xdef\f@baselineskip@normalsize{\f@baselineskip}
}
\newcommand{\smalltonormalsize}{%
  \fontsize
    {\fpeval{(\f@size@small+\f@size@normalsize)/2}}
    {\fpeval{(\f@baselineskip@small+\f@baselineskip@normalsize)/2}}%
  \selectfont
}
\makeatother

\ifpdf
\hypersetup{
	pdftitle={Algebraic multigrid methods for metric-perturbed coupled problems},
	pdfauthor={A. Budi\v{s}a, X. Hu, M. Kuchta, K.-A. Mardal and L. T. Zikatanov}
}
\fi

\title{Algebraic multigrid methods for metric-perturbed \\ coupled problems \thanks{Submitted to the editors DATE.
    \funding{AB and KAM acknowledge the financial support funded by the
	Norwegian Research Council grant 102155. The work of XH is partially supported by the National Science Foundation under grant DMS-2208267. MK acknowledges support from Norwegian Research Council grant 303362. The research of LTZ is supported in part by the U. S.-Norway Fulbright Foundation and the U. S. National Science Foundation grant DMS-2208249.}}}

\author{
	Ana Budi\v{s}a\thanks{ Simula Research Laboratory, Kristian Augusts gate 23, 0164 Oslo, Norway, (\email{ana@simula.no}, \email{miroslav@simula.no}, \email{kent-and@simula.no}). }
	\and Xiaozhe Hu\thanks{ Department of Mathematics, Tufts University, 177 College Ave, Medford, MA 02155, USA, (\email{xiaozhe.hu@tufts.edu}). }
	\and Miroslav Kuchta\footnotemark[2]
	\and Kent-Andre Mardal\footnotemark[2] \thanks{ Department of Mathematics, University of Oslo, P.O. Box 1053, Blindern, 0316, Oslo, Norway (\email{kent-and@math.uio.no}). }
	\and \newline Ludmil Tomov Zikatanov\thanks{ Department of Mathematics, Penn State, 239 McAllister Building, University Park, PA 16802, USA, (\email{ludmil@psu.edu}). }
}

\headers{AMG for metric-perturbed coupled problems}{A. Budi\v{s}a, X. Hu, M. Kuchta, K.-A. Mardal and L. T. Zikatanov}

\begin{document}
	
	\maketitle
	
	\begin{abstract}
	  We develop multilevel methods for interface-driven multiphysics problems that can be coupled across dimensions and where complexity and strength of the interface coupling deteriorates the performance of standard methods. We focus on solvers based on aggregation-based algebraic multigrid methods with custom smoothers that preserve the coupling information on each coarse level. We prove that with the proper choice of subspace splitting we obtain uniform convergence in discretization and physical parameters in the two-level setting. Additionally, we show parameter robustness and scalability with regards to number of the degrees of
          freedom of the system on several numerical examples related to the biophysical processes in the brain,
          namely the electric signalling in excitable tissue modeled by bidomain,
          EMI and reduced EMI equations.
	\end{abstract}
	
	\begin{keywords}
		algebraic multigrid method, preconditioning, iterative method, coupled problems, graph Laplacian
	\end{keywords}
	
	\begin{MSCcodes}
		65F08, 65N55, 65S05
	\end{MSCcodes}
	

	\section{Introduction} \label{sec:introduction}

    In this paper we will consider multilevel methods for a family of problems of the 
    form: 
	Find $ u \in V(\Omega) $ such that
	\begin{equation} \label{abstract}
		\left( A_\Omega + \gamma R_\Gamma' R_\Gamma \right) u = f. 
	\end{equation}
	Here $A_\Omega: V(\Omega) \rightarrow V(\Omega)'$ is an elliptic operator, while $ R_\Gamma' R_\Gamma $, for $ R_\Gamma : V(\Omega) \to V (\Gamma)' $ and $ \Gamma \subset \overline{\Omega} $, represents a lower order and singular term that is strongly weighted ($\gamma \gg 1$). We will refer to it as the \emph{metric term} as by assumption $R_\Gamma' R_\Gamma$ is a symmetric and
	semi-definite operator. In fact, $ R_\Gamma $ are typically either identity or projection operators which in the limit $\gamma\rightarrow\infty$ together enforce coupling either in the whole domain or on parts of it. As $A_\Omega$ is elliptic, multilevel methods for these operators are readily available as solvers, but performance is typically lost for large $\gamma$. Specifically, the topic of this paper is to adapt the smoothers to obtain robustness in $\gamma$. 

        
        The abstract problem arises in many multicompartment, multiphysics, and multiscale applications. 
        For multicompartment problems, a common approach is to consider the system in terms of its blocks 
        and adapt appropriate block preconditioners. Examples are 
        the bidomain equations~\cite{sundnes2007computing, pavarino2008multilevel} in cardiac modeling and
        multiple-network poroelasticity problems~\cite{hong2019conservative, hong2020parameter, corti2022numerical, piersanti2020parameter, piersanti2021parameter} in porous media modeling. The block approach has been quite successful and in most 
        situations parameter-robust solution algorithms have been found. Arguably a more challenging family of 
        problems are the multiphysics problems coupled through a common interface, which is a manifold of codimension 1.
        For some, but not all of these problems,
        metric terms at the interface arise. Examples are the so-called 
        EMI (Extracellular-Membrane-Intracellular) model of excitable tissue
        \cite{agudelo2013computationally, stinstra2009comparison, tveito2017cell-1} or the Biot-Stokes
        coupled problems~\cite{boon2022parameter}.
        Finally, certain multiscale problems are interface coupled problems in which dimensionality of 
        one of the problems is reduced by model reduction techniques. Here, the examples are the 3D-1D problem of
        tissue perfusion~\cite{d2008coupling}
        or well-block pressure in reservoir simulations~\cite{peaceman1978interpretation}. In particular for tissue perfusion modeling of e.g. whole-brain vasculature corresponding to tens of millions to tens of billions vessels in mice and humans~\cite{linninger2019mathematical}, respectively, simulations are a major challenge  and 1D representation for the vascular networks seems like a reasonable assumption.

        Multigrid methods for singularly perturbed problems have been considered in several settings. Examples include discretizations of the linear elasticity equations in primal form~\cite{schoberl1999multigrid} or $H(\text{div})$ and $H(\text{curl})$ problems~\cite{arnold2000multigrid}.
        Furthermore, the methods were generalized to algebraic multilevel methods (AMG) in \cite{lee2007robust}. A crucial 
        observation is that kernel or near kernel must be carefully treated or else the performance of the method deteriorates when the coupling parameter $\gamma$ increases. 
        However, the main challenge for the problems considered here is to properly capture 
        the metric term. In practice, the term can be localized to a part of the domain, and is possibly represented on different meshes or by discretizations that do not necessarily conform to each other. 
        This type of systems have been studied previously in \cite{M3AS_singular, MathComp_singular} and were described as nearly-singular systems. In this paper, we build on the those results to establish the uniform convergence with respect to both the coupling parameter $ \gamma $ and discretization parameters of the two-level AMG method for solving metric-perturbed problems as in \eqref{abstract}.

        
	The paper is organized as follows. After \Cref{sec:examples} where
        motivatory applications are presented we state our main results
        in \Cref{sec:two-level-AMG}. Experimental results showcasing robustness
        of the developed multgrid method are given in \Cref{sec:implementation}.
        We finally draw conclusions in \Cref{sec:conclusion}.

	\section{Motivatory examples} \label{sec:examples}
        To motivate the computational method developed in this paper 
        we first provide several practical examples which fit the template
        of the abstract problem \eqref{abstract}.
	
	\subsection{Bidomain model} \label{sec:bidomain}
	An example of a multicompartment problem is the so-called bidomain equations used
 to model the electrical activity of the heart~\cite{tung1978bi}. It is a system of nonlinear ordinary (ODE) and partial (PDE) differential equations  typically solved using an operator-splitting approach to solve ODE and PDE parts separately, cf. the overviews 
 \cite{franzone2014mathematical, sundnes2007computing}. Then, at each PDE time step one seeks $ u_e : \Omega \to \RR $ and $u_i : \Omega \to \RR  $, such that 
	\begin{subequations}
		\label{eq:bidomain}
		\begin{align}
			- \nabla \cdot (\alpha_e \nabla u_e) + \gamma (u_e - u_i) & = f_e & \text{ in } \Omega, \\
			- \nabla \cdot (\alpha_i \nabla u_i) + \gamma (u_i - u_e) & = f_i & \text{ in } \Omega,
		\end{align}
	\end{subequations}
where, the unknowns $u_e$ and $u_i$ are the extracellular and intracellular potentials, respectively. Here,  $\Omega$ is the tissue and $\gamma$ relates inversely to the time step size, with suitable boundary conditions assigned. 
Efficient methods for the formulation of bidomain equations in terms of $u_i$ or $u_e$ and
the so-called transmembrane potential $u_i-u_e$ have been developed by e.g. ~\cite{dos2004parallel,pennacchio2009algebraic, sundnes2006computational, huynh2022scalable, zampini2014dual}. Here we focus on formulation \eqref{eq:bidomain} with unknown intra- and extracellular potentials.
	
	To solve the equations \eqref{eq:bidomain}, we discretize the system using finite element method (FEM). Denote with $L^2 = L^2(\Omega)$ the function space of square-integrable functions on $ \Omega $ and $ H^s = H^s(\Omega) $ the Sobolev spaces with $s$ derivatives in $ L^2 $. Furthermore, let $ V \subset H^1(\Omega) $ be the discretization by continuous linear finite elements ($ \mathbb{P}_1 $). 
	The discrete variational formulation states to find $ u_e, u_i \in V $ such that for $ f_e, f_i \in V' $
	\begin{equation} \label{eq:bidomain-problem}
		\left(
		\begin{pmatrix}
			- \alpha_e \Delta &  \\
			& \alpha_i \Delta
		\end{pmatrix}
		+
		\gamma 
		\begin{pmatrix}
			I & -I \\
			-I & I
		\end{pmatrix}
		\right)
		\begin{pmatrix}
			u_e \\ u_i
		\end{pmatrix}
		=
		\begin{pmatrix}
			f_e \\ f_i
		\end{pmatrix}
	\end{equation}
	we see that for $ \gamma > 0 $, the system is symmetric positive definite (SPD). 
 However, it contains a singular strongly weighted lower order term for which the kernel functions $ (v_e, v_i) \in V \times V $ are such that $ v_e = v_i $. It is clear that 
 the kernel contains both high and low frequency components, making it critical to handle the kernel in the multilevel algorithm~\cite{lee2007robust}.
 
	

	\subsection{EMI model} \label{sec:emi}
        The modeling assumption of co-existence of the interstitium, extracellular space and
        the cell membrane which is at the core of the bidomain system \eqref{eq:bidomain} has 
        recently been challenged by the EMI models \cite{agudelo2013computationally, tveito2017cell-1} (also known as cell-by-cell models \cite{huynh2022convergence}).
        Here, the geometry of each of the compartments is resolved explicitly leading to a coupled
        mixed-dimensional problem posed on $D$-dimensional domains $\Omega_i\subset\Omega_e$ seperated
        by the inteface $\Gamma=\overline{\partial\Omega_i}\cap\overline{\partial\Omega_e}$ which
        is a manifold of codimension 1. Following the operator splitting approach as in
        \Cref{sec:bidomain} the PDE step now solves
	\begin{subequations}
		\label{eq:emi_2d}
		\begin{align}
                  - \nabla \cdot (\alpha_e \nabla u_e) &= 0 & \text{ in } \Omega_e, \\
		  - \nabla \cdot (\alpha_i \nabla u_i) &= 0 & \text{ in } \Omega_i, \\
                  \alpha_i\nabla u_i \cdot\nu_i + \alpha_e\nabla u_e\cdot\nu_e &= 0 &\text{ on }\Gamma,\\
                  \gamma(u_i - u_e) + \alpha_i\nabla u_i\cdot\nu_i &= f &\text{ on }\Gamma.
		\end{align}
	\end{subequations}
        Here $f$ is the source term coming from the ODE part and $\nu_{\iota}$ the normal
        vector on $\Gamma$ pointing outwards with respect to $\Omega_{\iota}$, $\iota\in\left\{i, e\right\}$.
        The system is typically equipped with homogeneous Neumann conditions on $\partial\Omega_e$. We remark that in
        \eqref{eq:emi_2d} we assumed, for simplicity, that $\Omega_e$ contains only a single
        cell/intracellular domain.

        Variational formulation of \eqref{eq:emi_2d} posed in $V_e\times V_i$ with 
        $V_e=H^1(\Omega_e)$, $V_i=H^1(\Omega_i)$ gives rise to a problem
        \begin{equation} \label{eq:emi_2d_operator}
          \left(
			\begin{pmatrix}
			  -\alpha_e\Delta & \\
                          & -\alpha_i\Delta \\
			\end{pmatrix}
                        +
                        \gamma
                        \begin{pmatrix}
                          \tau'_e\tau_e &  -\tau_e'\tau_i\\
                          -\tau'_i\tau_e & \tau'_i\tau_i\\
                        \end{pmatrix}
                        \right)
			\begin{pmatrix}
				u_{e} \\ u_{i}
			\end{pmatrix}
		=
			\begin{pmatrix}
				f_e \\ f_i
			\end{pmatrix},
	\end{equation}
        where the perturbation involves trace operators $\tau_{\iota}$ such that
        $\tau_{\iota}v_{\iota}=v_{\iota}|_{\Gamma}$ for $\iota\in\left\{i, e\right\}$ and
        $v_{\iota}$ a continuous function on the respective domains. We observe
        that the EMI model \eqref{eq:emi_2d} formulated in terms of intra-/extracellular
        potentials takes the form of the abstract problem \eqref{abstract} where
        in particular, 
        the perturbation operator here induces $(u_e, u_i)\mapsto \int_{\Gamma}(u_i-u_e)^2$.
        Robust domain-decomposition solvers for \eqref{eq:emi_2d_operator} have recently
        been developed in \cite{huynh2022convergence}.
        
	\subsection{Reduced 3D-1D EMI model} \label{sec:3d1d}
	In a number of applications in geoscience (e.g. resorvoir simulations) and biomechanics (e.g. microcirculation)
        the EMI model \eqref{eq:emi_2d} is applied in geometrical setups where the domain
        $\Omega_i$ is large but slender such that its resolution as a 3D structure by a computational
        mesh is impractical. This issue is addressed by model reduction which results in a 1D representation
        of $\Omega_i$ by a smooth (centerline) curve $\Gamma$.
        In \cite{d2008coupling} a mathematical formulation of a 3D-1D coupled problem with application to
        tissue perfusion was analyzed. The numerical approximation of the problem has been a topic of many
        subsequent works, see e.g. \cite{gjerde2020singularity,hodneland2021well,laurino2019derivation,koppl2018mathematical,kuchta2021analysis} 
        as it includes in particular the challenge of traces of co-dimension 2 for standard elliptic problems
        which are not well defined on $H^1(\Omega)$, $\Omega=\Omega_i\cup\Omega_e$.
        Relatively few works have considered preconditioners for such problems
        \cite{hu2023effective, kuchta2019preconditioning, cerroni2019mathematical}.
        
        To fit into the abstract setting of \eqref{abstract} we here consider
        a reduced EMI problem \cite{laurino2019derivation}: Find $u_e\in H^1(\Omega)$, $u_i\in H^1(\Gamma)$ such
        that 
        \begin{equation} \label{eq:emi_3d_operator}
          \left(
			\begin{pmatrix}
			  -\alpha_e\Delta & \\
                          & -\alpha_i\Delta \\
			\end{pmatrix}
                        +
                        \gamma
                        \begin{pmatrix}
                          \Pi_{\rho}'\Pi_{\rho} &  -\Pi_{\rho}'\\
                          -\Pi_{\rho} & I\\
                        \end{pmatrix}
                        \right)
			\begin{pmatrix}
				u_{e} \\ u_{i}
			\end{pmatrix}
		=
			\begin{pmatrix}
				f_e \\ f_i
			\end{pmatrix}.
	\end{equation}
        Here $\Pi_{\rho}$ is the \emph{averaging} operator reducing $u\in H^1{(\Omega)}$
        to $\Gamma$ by computing the function's average over a virtual cylinder
        with radius $\rho$ which approximates the domain $\Omega_i$. More precisely,
        we let 
	\begin{equation}\label{eq:average_op}
	  \left(\Pi_{\rho} u\right)(x) = \frac{1}{\lvert C^{\nu}_\rho(x) \rvert} \int_{C^{\nu}_\rho(x)} u \quad u\in H^1{(\Omega)},
	\end{equation}
        where $x\in\Gamma$ and $C^{\nu}_\rho(x)$ is a circle of radius $\rho(x)$ in the plane
	with a normal $\nu=\tfrac{\mathrm{d}\Gamma}{\mathrm{d}s}(x)$ and $s$ being the arc-length
        coordinate of $\Gamma$. Furthermore, for a smooth function $v$ on $\Gamma$ we define
        $\Delta v = \tfrac{\mathrm{d^2} v }{\mathrm{d}s^2}$.
        
        We remark that the perturbation operator in \eqref{eq:emi_3d_operator}
        is symmetric while the formulations \cite{gjerde2020singularity, d2008coupling}
        utilize also a standard trace operator in the coupling (in addition to $\Pi_{\rho}$)
        leading in turn to a non-symmetric coupling term. Let us finally stress that \eqref{eq:emi_3d_operator}
        is typically only a component in advanced models that include convection and other processes, cf.
        \cite{goirand2021network,hartung2021mathematical}.

	\section{Two-level AMG for metric-perturbed coupled problems} \label{sec:two-level-AMG} 
		
	In this section, we first reformulate the example systems of PDEs into a more general setting. This allows us to introduce aggregation-based AMG methods, a general class of methods that can be used to solve a wide variety of PDE systems. We then prove the uniform convergence of a two-level AMG method under certain assumptions on the underlying subspace decomposition. These assumptions are shown to be sufficient and suitable for the problems we consider.
	 
	
	\subsection{Preliminaries} \label{sec:preliminaries}
    Let $ \Omega \subset \mathbb{R}^{d_{\Omega}} $, $ \Upsilon \subset \mathbb{R}^{d_{\Upsilon}} $ and $ \Gamma \subset \mathbb{R}^{d_{\Gamma}} $ such that $ 0 < d_{\Gamma} \leq d_{\Omega}, d_{\Upsilon} \leq 3 $, $ \overline{\Omega}\cap \overline{\Upsilon}\neq\emptyset $ and $\overline{\Gamma} \subset \overline{\Upsilon} $. On each of the domains we introduce quasi-uniform triangulation and a corresponding finite element space $V_i$, $i\in\left\{\Omega, \Upsilon, \Gamma \right\}$ with $V_\Gamma \subseteq V_\Upsilon$. 
    For $ V = V_\Omega \times V_\Upsilon $ and we then consider bilinear forms $ a_0(\cdot, \cdot), a_1(\cdot, \cdot) : V \times V \to \RR $ defined as follows,
    \begin{subequations} \label{eq:bilinear-forms}
       	\begin{align}
       		a_0((u_\Omega, u_{\Upsilon}), (v_\Omega, v_\Upsilon)) & = m_{\Gamma}(R (u_\Omega, u_{\Upsilon}) , R (v_\Omega, v_\Upsilon)), \\
       		a_1((u_\Omega, u_\Upsilon), (v_\Omega, v_{\Upsilon})) & = a_\Omega(u_\Omega, v_\Omega) + a_{\Upsilon}(u_{\Upsilon \backslash \Gamma}, v_{\Upsilon \backslash \Gamma}) + a_\Gamma(u_\Gamma, v_\Gamma).
       	\end{align}
    \end{subequations}
    for $ v_\Upsilon = (v_{\Upsilon\backslash\Gamma}, v_\Gamma) $ and $v_\Gamma \in V_\Gamma$. Here, $ a_\Omega(\cdot, \cdot) $, $ a_{\Upsilon}(\cdot, \cdot) $ and $ a_{\Gamma}(\cdot, \cdot) $ are bilinear forms corresponding to the elliptic equations, such as $ d_\Omega $-, $d_\Upsilon$- and $d_\Gamma$-Laplacians on their respective domains. The bilinear form $ m_{\Gamma}(\cdot ,\cdot)$ is a lower-order (mass) term in $ V_\Gamma $. The interface operator $ R : V \to V_\Gamma' $ defines a metric on the interface, that is for $ v = (v_\Omega, v_\Upsilon ) \in V $,
    \begin{equation}\label{eq:interface_op}
       	R v = v_\Gamma - \sigma(v_\Omega),
    \end{equation}
    where $ \sigma : V_\Omega \to V_\Gamma' $ is a linear restriction operator.
    In particular, we assume
    $ \| \sigma(v_\Omega) \|_{L^2(\Gamma) } \lesssim \| v_\Omega \|_{V_\Omega} $.
    The main problem we want to solve is to find $ u \in V $ such that
    \begin{equation} \label{eq:general_problem}
    	a(u, v) = \gamma a_0(u, v) + a_1(u, v) = f(v) \qquad \forall v \in V,
    \end{equation}
    where $f \in V'$ and $ \gamma \gg 1 $ is a coupling parameter.

	Finally, we can define operators representing the bilinear forms in \eqref{eq:bilinear-forms}. Let $ A, A_0, A_1 : V \to V' $ such that $ \langle A_0 u, v \rangle = a_0(u, v) $, $ \langle A_1 u, v \rangle = a_1(u, v) $ and $ A = \gamma A_0 + A_1 $ for $ u, v \in V $. Here $ \langle \cdot, \cdot \rangle $ is the duality pairing between $ V $ and its dual $ V' $. Additionally, let $ \| v \|_{\tilde{A}}^2 = \langle \tilde{A} v, v \rangle $ for any symmetric positive definite operator $ \tilde{A} $ on $ V $ and $ v \in V $. Equivalently to \eqref{eq:general_problem} we want to find $ u \in V $ such that
	\begin{equation} \label{eq:matrix-general-problem}
		\langle A u, v \rangle = \gamma \langle A_0 u, v \rangle + \langle A_1 u, v \rangle = \langle f, v \rangle \qquad \forall v \in V.
	\end{equation}
	We refer this system as the \emph{metric-perturbed coupled problem}, since it is perturbed by a lower-order term $ A_0 $ that can dominate when $ \gamma \gg 1 $.

	\begin{remark} \label{remark:about-R}
		In general, the interface operator \eqref{eq:interface_op} can be represented as
		$
		R = \begin{pmatrix}
			-\sigma_{\Omega} & \sigma_{\Upsilon}
		\end{pmatrix}
		$
		where $ \sigma_{\Omega} $ and $ \sigma_{\Upsilon} $ are linear restriction operators (trace or averaging) on $ \Gamma = \overline{\Omega} \cap \overline{\Upsilon} $. This generality would represent the case of non-conforming meshes between each subdomain $ \Omega $ and $\Upsilon$ and their interface $ \Gamma $. For example, in the EMI model \eqref{eq:emi_2d}, $ \sigma_{i} $ are the respective trace operators $\sigma_{\iota}(v)=v|_{\Gamma}$, $\iota\in \{\Omega, \Upsilon\}$. However, we assume that at least on of the triangulations of subdomains, namely $ \Upsilon $, conforms to the interface (such that $ V_\Gamma \subseteq V_\Upsilon $) and the restriction operator becomes of form $\sigma_\Upsilon = \begin{pmatrix}
			I_\Gamma & 0_{\Upsilon \backslash \Gamma}
		\end{pmatrix}. $
		
		That implied, we see that the subdomain part $ \Upsilon \backslash \Gamma $ does not contribute to the metric coupling term $ a_0(\cdot, \cdot) $. Therefore, the component of functions in $V_\Upsilon $ defined only on $ \Upsilon \backslash \Gamma $ will not influence the convergence of the AMG method with regards to parameter $ \gamma $ and can be smoothed using standard methods, such as Gauss-Seidel or Jacobi method.
		
		Hence, to simplify the exposition of the convergence theory, we will consider only the case when $ V_\Gamma = V_\Upsilon $ further in this paper.
		For example, in the reduced $3D$-$1D$ EMI model \eqref{eq:emi_3d_operator}, we have that $ \Upsilon = \Gamma $ is a curve in $ \Omega\subset\mathbb{R}^3 $, $ \sigma_{\Omega} $ is the averaging operator \eqref{eq:average_op} while $ \sigma_{\Upsilon}$ is the identity map.
		Note that for the bidomain model \eqref{eq:bidomain} we have $ \Omega = \Upsilon = \Gamma $ and the restriction operators are simply identities.
	\end{remark}
	
	\begin{remark} \label{remark:about-bound-on-A}
		By assumptions of elipticity of $A_1$ and boundedness of the restriction operators in the $A_1$-induced norm on $V$ the
		following equivalence holds
		\[
		\lVert u \rVert_{A_1}^2 \leq \langle Au, u \rangle \lesssim \gamma \lVert u 		\rVert_{A_1}^2\quad \forall u\in V
		\]
		and we observe that the upper bound depends on $\gamma$. That is, preconditioning strategies based on the block-diagonal operator $A_1$ (e.g. AMG with point smoothers) cannot be robust in the coupling parameter.
	\end{remark}

	In the following, we slightly abuse the notation and consider $ A $, $ A_1 $ and $ A_0 $ to be matrices and $u_{\Omega}$ and $u_{\Gamma}$ to be vectors that we obtain from a choice of a FEM basis, such as the linear continuous finite elements, i.e. $ \mathbb{P}_1 $ elements. We will present theoretical results based on reformulating the problem in terms of graph Laplacians. To do so, we introduce the undirected graph $ \mathcal{G}(A) $ associated with the sparsity pattern of the symmetric positive definite matrix $ A $. The vertices of $ \mathcal{G}(A) $ are labeled as $ \mathcal{V} = \{ 1, 2, \dots, N \} $, where $N$ represents the number of degrees of freedom (DOFs) of $ V $. We use $ \mathcal{E} $ to denote the collection of edges $ e = (i, j) $ if $(A)_{ij} \neq 0 $, with an intrinsic ordering. Specifically, we order any graph edge $ e =(i, j) $ with $ j < i $ for $ i, j \in \mathcal{V} $. It is worth noting that the following equivalences are well-known:
	\begin{enumerate}
		\item In the context of finite element methods, mass matrices represent the matrix form of the $ L^2 $-inner product on a portion or the entirety of the domain, and are equivalent to diagonal matrices, as shown in \cite{wathen1987realistic}. For instance, this equivalence holds for $ \mathbb{P}_k$ elements, where the constants depend only on the polynomial order $k$. On the other hand, stiffness matrices, which correspond to second-order elliptic operators such as $ A_1 $, are spectrally equivalent to weighted graph Laplacians. We provide a sketch of the proof in \Cref{sec:apx-graph-laplacians} using results from \cite[Lemma~14.1]{2017XuZikatanov-aa}.
		\item If $\widetilde{A} : V \to V'$ is any positive semidefinite matrix and $\widetilde{D}$ its diagonal, then we have
		\begin{equation}\label{eq:A-le-D}
			\| v \|_{\widetilde{A}}^2 \lesssim \| v \|_{\widetilde{D}}^2 \qquad v \in V.
		\end{equation}
		The constants hidden in this estimate depend on the number of nonzeroes per row in $\widetilde{A}$. The estimate is easily derived using the Schwarz inequality. 
		\item In particular, for any graph Laplacian on $V$, the following \emph{Poincar{\'e} inequality} holds:
		\begin{equation} \label{eq:graph-poincare}
			\inf_{c \in \RR} \| v - c \mathbbm{1} \|_{\widetilde{D}}^2  \lesssim \| v \|_{\widetilde{A}}^2 \qquad v \in V.
		\end{equation}
		The constants are determined by the weights in $ \widetilde{A} $ and are proportional to the square of the number of vertices in the graph divided by the square of the size of the minimal cut in the graph. The complete proof can be found, for example, in~\cite{1989JerrumM_SinclairA-ab}. This result is used only locally for small size graphs, namely on each aggregate that represents the coarse scale degree of freedom in the two-level AMG method. 
	\end{enumerate}
	
	Consequently, the bilinear forms from \eqref{eq:bilinear-forms} can be replaced by their equivalent graph forms. Let $ \mathcal{V} = \mathcal{V}_\Omega \cup \mathcal{V}_\Gamma $ be the division of graph vertices into two subsets with regards to discretizations of $ \Omega $ and $ \Gamma $, respectively. Similarly, let $ \mathcal{E} = \mathcal{E}_\Omega \cup \mathcal{E}_\Gamma $. Then for $ u = (u_\Omega, u_\Gamma) \in V $ and $ v = (v_\Omega, v_\Gamma) \in V $ we get
	\begin{subequations} \label{eq:graph-bilinear}
		\begin{align}
			a_{\Omega}(u_\Omega, v_\Omega) & \eqsim \sum_{e \in \mathcal{E}_{\Omega}} \omega_e \, \delta_e u_\Omega \, \delta_e v_\Omega \label{eq:graph-bilinear-a} \\
			a_{\Gamma}(u_\Gamma, v_\Gamma) & \eqsim \sum_{e \in \mathcal{E}_{\Gamma}} \omega_e \, \delta_e u_\Gamma \, \delta_e v_\Gamma \label{eq:graph-bilinear-b} \\
			a_0(u, v) & \eqsim \sum_{k = 1}^{N_\Gamma} m_k \left( u_{\Gamma, k} - (\sigma(u_\Omega))_k \right) \left( v_{\Gamma, k} - (\sigma(v_\Omega))_k \right) \label{eq:graph-bilinear-c} \\
			\delta_e v = v_i - v_j,  \quad e &= (i,j), \quad \omega_e = \omega_{ij} > 0, \quad j<i, \nonumber
		\end{align}
	\end{subequations}
	where $N_\iota = \dim V_\iota $, $\iota \in \{\Omega, \Gamma \}$. The weights $ \omega_e $ depend on the shape regularity of the mesh and their behavior is as $ h^{d-2} $
	if they correspond to a $ d $-homogeneous simplicial complex. The elements $ m_k $ behave like $ h^{d_\Gamma} $ if $ \Gamma $ corresponds to a $ d_\Gamma $-homogeneous simplicial complex. In the following section, we design aggregation-based AMG method with special Schwarz smoothers to solve \eqref{eq:matrix-general-problem}. For that, we use the above equivalences for the bilinear forms to prove that this AMG method satisfies the kernel and stability conditions that guarantee uniform convergence. 
	
	
	\subsection{Convergence of the two-level AMG} \label{sec:convergence-AMG}
	The main idea of any algebraic multigrid (AMG) method is to construct a hierarchy of nested vector spaces, each of which targets different error components for the solution of \eqref{eq:matrix-general-problem}. In the case of aggregation-based AMG methods, such as unsmoothed aggregation AMG (UA-AMG) and smoothed aggregation AMG (SA-AMG), there is an added advantage in that multiple approximations of near-kernel components of the matrix describing the linear system can be retained as elements of each subspace in the hierarchy. To illustrate this, we first introduce the necessary ingredients of AMG in the context of subspace correction methods \cite{M3AS_singular, Xu1992SIAMReview, XuZikatanov2002}.
	
	Let us introduce the decomposition $ V = V_c + \sum\limits_{j = 1}^J V_j $ where $V_c \subset V$ and $ V_j \subset V $, $j=1,\cdots,J$. Then the AMG preconditioner associated with such subspace splitting for the system \eqref{eq:matrix-general-problem} is defined as
	\begin{subequations} \label{eq:preconditioner}
		\begin{align}
			B & = P_c + S, \quad S = \sum_{j=1}^J P_j,   \text{ where } \label{eq:preconditioner-a}\\
			\left\langle S^{-1} w, w \right\rangle  & = 
			\inf \left\{ \sum_{j = 1}^J \| w_j \|_A^2 \, : w = \sum_{j=1}^J w_j \text{ and } w_j \in V_j, \, j = 1, \dots, J \right\}, \label{eq:preconditioner-b}\\
			\left \langle B^{-1} v, v \right\rangle  & = 
			\inf \left\{\| v_c \|_A^2 + \sum_{j = 1}^J \| v_j \|_A^2 \, : v = v_c + \sum_{j=1}^J v_j \text{ and } v_c \in V_c,\, v_j \in V_j, \, 1, \dots, J \right\}, \label{eq:preconditioner-c}
		\end{align}
	\end{subequations}
	where $ P_j $ are the $A$-orthogonal projections on $ V_j $ for $ j = 1, \dots, J $. Here, $ V_c $ accounts for the correction on a coarse (sub)space, while $ V_j $ for $ j \geq 1 $ define a Schwarz-type smoother on the fine grid. 
	
	Choosing the appropriate subspace decomposition is the essence of a robust and efficient preconditioner for the system  \eqref{eq:matrix-general-problem}. Therefore, we want to show that, within certain assumptions, $B$ is a uniform preconditioner for $ A $ with regards to $ \gamma $ and $ h $. The assumptions required in the convergence analysis are as follows.
	\begin{enumerate}[(I)]
		\item \textbf{Kernel decomposition condition}: Find the subspace decomposition $ V_j $ for $ j = 1, \dots, J $, such that
		\begin{equation} \label{eq:kernel-condition}
			\Ker(A_0) = \Ker(A_0) \cap V_c + \sum_{j = 1}^{J} \Ker(A_0) \cap V_j.
		\end{equation}
		\item \textbf{Stable decomposition condition}: 
		For a given $ v \in V $, there exist a splitting $ \{ v_c \} \cup \{ v_j \}_{j = 1}^J $, $v_c \in V_c$ and $\, v_j \in V_j $ such that
		\begin{equation} \label{eq:stable-condition}
			\| v_c \|_{A_1}^2 + \sum\limits_{j = 1}^J \| v_j \|_{A_1}^2 \lesssim \| v \|_{A_1}^2, \quad
			| v_c |_{A_0}^2 + \sum\limits_{j = 1}^J | v_j |_{A_0}^2 \lesssim | v |_{A_0}^2.
		\end{equation}
	\end{enumerate}
	Specifically, the subspace splitting of $ V $ defines a Schwarz preconditioner that uniformly bounds the condition number in $ \gamma $ if the kernel condition \eqref{eq:kernel-condition} is satisfied. Additionally, the uniform bound in $ h $ is guaranteed if for any $v \in V $ we construct an aggregate decomposition (coarse grid) stable in $\|\cdot\|_{A_1}$ that is also stable in $| \cdot |_{A_0}$.

\begin{remark} \label{remark:kernel-for-piecewise}
The kernel decomposition condition \eqref{eq:kernel-condition} usually fails for pointwise smoothers in the presence of metric terms. As illustration, we consider the bidomain equations where the kernel consists of 
functions of the form $u_e(x) = u_i(x)$, $ x \in \Omega $. 
If $e_j$ is the $j$-th coordinate vector, then a point smoother only corrects in the one dimensional space $\operatorname{span}\{e_j\}$. Clearly, $\Ker(A_0)\cap\operatorname{span}\{e_j\} = \{0\}$ as 
$\Ker \left(A_0\right) =
\operatorname{Range}\left\{\begin{pmatrix}
I  \\  I
\end{pmatrix}\right\}$. This shows that the condition \eqref{eq:kernel-condition} does not hold and as is seen from the numerical tests presented later the method with point-wise smoother is far from optimal with respect to  the size of the coupling term $\gamma$. 
\end{remark}

Assuming these two conditions hold, we first show the main results on the condition number estimate of the system \eqref{eq:matrix-general-problem} preconditioned with $ B $ \eqref{eq:preconditioner}. More precisely, we show that the condition number $ \kappa(BA) $ of the preconditioned system \eqref{eq:matrix-general-problem} is bounded uniformly with respect to $ \gamma $ and $h$. Note that the orthogonal complement of $ \Ker(A_0) $ can be defined as follows,
	\begin{align} \label{eq:orthogonal-compl-kerA0}
		\Ker(A_0)^\perp & = \{ y \in V : \langle A y, z \rangle = 0, \, \forall z \in \Ker(A_0)  \} \nonumber \\
						& = \{ y \in V : \langle A_1 y, z \rangle = 0, \, \forall z \in \Ker(A_0)  \},
	\end{align}
	and similarly we can define local kernels $ \Ker(A_0) \cap V_j $ and kernel complements $ \Ker(A_0)^\perp \cap V_j $, for $ j = c, 1, 2, \dots, J $. Furthermore, we define projections to local subspaces, that is for any $ v \in V $, let $ P_j : V \to V_j $, $ P_{1, j} : V \to V_j $, and $ P_{0, j} : V \to \Ker(A_0)^\perp \cap V_j $ such that for all $ w_j\in V_j $
		\begin{align} \label{eq:projections}
			\langle A (P_j v), w_j \rangle & = \langle A v, w_j \rangle, \nonumber \\
			\langle A_1 (P_{1,j} v), w_j \rangle & = \langle A_1 v, w_j \rangle, \\
			\langle A_0 (P_{0,j} v), w_j \rangle & =  \langle A_0 v, w_j \rangle. \nonumber
		\end{align}
	for $ j = c, 1, 2, \dots, J $. With that defined, we can easily derive that for $ v \in V $ and $ v = y + z $, $ z \in \Ker(A_0)$ and $ y \in \Ker(A_0)^\perp $ it follows
	\begin{equation} \label{eq:direct-sum-kernel-perp}
		\| v \|_A^2 = \| y \|_A^2 + \| z \|_{A_1}^2.
	\end{equation}
	The next lemma shows a triangle inequality for the projections defined in \eqref{eq:projections}.
	\begin{lemma} \label{lemma:projections-triangle-ineq} The following inequality holds for all $ v \in V $ and $j = c, 1, 2, \dots, J$,
		\begin{equation}
			\|P_j v \|^2 \le \gamma
			\left| P_{0,j} v \right|_{A_0}^2 +
			\left\| P_{1,j} v \right\|_{A_1}^2.
		\end{equation}
	\end{lemma}
	\begin{proof}
		Since $ P_j v \in V_j $, by the definitions of of the projections we have 
		\begin{align*}
			\| P_j v \|_A^2 
			& = \langle A P_j v, P_j v \rangle = \langle A v, P_j v \rangle = \gamma \langle A_0 v, P_j v \rangle + \langle A_1 v, P_j v \rangle \\
			& = \gamma \langle A_0 P_{0,j} v, P_j v \rangle + \langle A_1 P_{1,j} v, P_j v \rangle  \\
			& \le \frac{\gamma}{2} \left( \langle A_0 P_{0,j} v, P_{0,j} v \rangle + \langle A_0 P_j v, P_j v \rangle \right)
			+ \frac{1}{2} \left( \langle A_1 P_{1,j} v, P_{1,j} v \rangle + \langle A_1 P_j v, P_j v \rangle \right) \\
			& = \frac{\gamma}{2} \left| P_{0,j} v \right|_{A_0}^2 + \frac{1}{2} \| P_{1,j} v \|_{A_1}^2 
			+ \frac{1}{2} \| P_j v \|_A^2
		\end{align*}
		Moving the last term on the right to the left hand side finishes the proof. 
	\end{proof}
	
	Using the previous lemma and definitions, we are ready to present the condition number estimate in the following theorem. 
	\begin{theorem} \label{thm:condition-number}
		Let $ v = y + z \in V $ such that $ z \in \Ker(A_0) $ and $ y \in \Ker(A_0)^\perp $ that is the orthogonal complement with regards to $ A_1 $-inner product. Assuming the conditions \eqref{eq:kernel-condition} and\eqref{eq:stable-condition} hold, we get the estimate
		\begin{equation} \label{eq:cond-number-estimate}
		C_1 \leq\frac{\left\langle B^{-1} v, v \right\rangle}{\| v \|_A^2} \le C_2:= 2 \left[ C_{\perp}(v) + C_{0}(v) \right],
		\end{equation}
		where the constants $C_1$ only depends on the maximum over the number of intersections between the subspaces $ V_j $, $ j = 1, \dots, J $ and the constants $ C_{\perp}(v) $ and $ C_0(v) $ are
		\begin{equation*} 
			\begin{aligned}
				C_{\perp}(v) &= \inf_{y_c + \sum_{j = 1}^J y_j = y} \left(
				\frac{|y_c|^2_{A_0} + \sum_{j = 1}^J |y_j|^2_{A_0}}{|y|^2_{A_0}} +
				\frac{\|y_c\|^2_{A_1} + \sum_{j = 1}^J \|y_j\|^2_{A_1}}{\|y\|^2_{A_1}}
				\right),
				\\
				C_{0}(v) &=
				\inf_{z_c + \sum_{j = 1}^J z_j = z}\frac{\|z_c\|^2_{A_1} + \sum_{j = 1}^J \| z_j \|^2_{A_1}}{\| z \|^2_{A_1}}.
			\end{aligned}
		\end{equation*}
	This implies that the condition number estimate $\kappa(BA) \leq C_2/C_1$.
	\end{theorem}
	
	\begin{proof}
		From the definition of $ B^{-1} $ in \eqref{eq:preconditioner-c}, if $ v \in V $ and $ v = y + z $ with
		$ y \in \Ker(A_0)^\perp $, $ z \in \Ker(A_0)$ we have that
		\begin{align*}
			\left\langle B^{-1} v, v \right\rangle
			& = \inf_{v_c + \sum_{j = 1}^J v_j=v} \left( \|v_c\|_A^2 + \sum_{j=1}^J \|v_j\|_A^2 \right) \\
               & \le 
			\inf_{y_c + \sum_{j = 1}^J y_j=y;\;z_c + \sum_{j = 1}^J z_j=z} \left( \|y_c + z_c \|_A^2 + 
  \sum_{j=1}^J \|y_j+z_j\|_A^2   \right)  \\
			&\le 2\inf_{y_c + \sum_{j = 1}^J y_j=y;\;z_c + \sum_{j = 1}^J z_j=z} \left( \|y_c\|^2_A + \|z_c\|^2_A +  \sum_{j=1}^J \left( \|y_j\|^2_A +
			\|z_j\|_A^2 \right) \right) \\
			& = 2 \inf_{y_c + \sum_{j = 1}^J y_j=y} \left( \|y_c \|^2_A + \sum_{j=1}^J \|y_j\|^2_A \right) +
			2 \inf_{z_c + \sum_{j = 1}^J z_j=z} \left( \|z_c\|^2_{A_1} + \sum_{j=1}^J \|z_j\|_{A_1}^2 \right).
		\end{align*}
		The first inequality above is a crucial inequality as it follows from: (1) the fact that the set of decompositions of $ v = y + z $ is larger than the set of decompositions of $ y $ and $ z $ (because any decomposition of $ y $ and $ z $ gives a decomposition of $ v $), and (2) the kernel decomposition assumption \eqref{eq:kernel-condition} without which we cannot have a decomposition of $ z = z_c + \sum_{j = 1}^J z_j $ with $ z_j \in \Ker(A_0) \cap V_j $, $j = c, 1, 2, \dots, J$. Then, by using \eqref{eq:direct-sum-kernel-perp} we see that 
		{\smalltonormalsize \begin{align*}
                \dfrac{\left\langle B^{-1} v,v\right\rangle}{\|v\|_A^2}
			= 
			\dfrac{\left\langle B^{-1} v,v\right\rangle}{\|y\|_{A}^2+\|z\|_{A_1}^2}
			& \le
			2\dfrac{ \displaystyle{ \inf_{z_c+\sum_{j=1}^J z_j=z}} \left( \|z_c\|^2_{A_1} + \sum_{j=1}^J \|z_j\|_{A_1}^2 \right)}{\|y\|_A^2+\|z\|^2_{A_1}}+
			2\dfrac{ \displaystyle{\inf_{y_c + \sum_{j = 1}^J y_j=y}} \left( \|y_c\|^2_A + \sum_{j=1}^J \|y_j\|_{A}^2 \right)}{\|y\|_{A}^2+\|z\|_{A_1}^2}\\
			& \le
			2\dfrac{ \displaystyle{\inf_{z_c+\sum_{j=1}^J z_j=z}} \left( \|z_c\|^2_{A_1} + \sum_{j=1}^J \|z_j\|^2_{A_1} \right)}{\|z\|^2_{A_1}}+
			2\dfrac{ \displaystyle{\inf_{y_c + \sum_{j = 1}^J y_j=y}} \left( \|y_c\|^2_A + \sum_{j=1}^J \|y_j\|^2_A \right)}{\gamma|y|^2_{A_0}+\|y\|_{A_1}^2}\\
			& \le
			2C_0(v) + 2\inf_{y_c + \sum_{j=1}^J y_j=y}\dfrac{ \|y_c\|^2_A + \sum\limits_{j=1}^J \|y_j\|^2_A}{\gamma|y|^2_{A_0}+\|y\|_{A_1}^2}.
		\end{align*}
            }%
		Notice that, since $y_j\in V_j$, we have $y_j=P_{j} y_j=P_{1,j}y_j$. We now introduce the following elementary inequality for $t_1, t_2 > 0$ and $s_1, s_2 >0$:
		\begin{equation*}
			\frac{\gamma t_1+s_1}{\gamma t_2 + s_2} =    
			\frac{\gamma t_1}{\gamma t_2 + s_2} +
			\frac{s_1}{\gamma t_2 + s_2} \le 
			\frac{t_1}{t_2} + \frac{s_1}{s_2}.
		\end{equation*}
		With this in hand, it follows from \Cref{lemma:projections-triangle-ineq} that:
		\begin{align*}
			\dfrac{\|y_c\|^2_A+\sum_{j=1}^J \|y_j\|^2_A}{\gamma|y|^2_{A_0}+\|y\|_{A_1}^2}
			& = \frac{\|P_c y_c\|^2_A+\sum_{j=1}^J \|P_j y_j\|^2_A}{\gamma|y|^2_{A_0}+\|y\|_{A_1}^2}\\
			& \le
			\dfrac{\gamma \left( |P_{0,c}y_c|^2_{A_0} + \sum_{j=1}^J |P_{0,j}y_j|^2_{A_0} \right) + \left( \|P_{1,c}y_c\|^2_{A_1} + \sum_{j=1}^J \|P_{1,j}y_j\|^2_{A_1} \right)}{\gamma|y|^2_{A_0}+\|y\|_{A_1}^2}\\
			& \le \dfrac{|P_{0,c}y_c|^2_{A_0} +\sum_{j=1}^J |P_{0,j}y_j|^2_{A_0}}{|y|^2_{A_0}}
			+\dfrac{ \|y_c\|^2_{A_1}+\sum_{j=1}^J \|y_j\|^2_{A_1}}{\|y\|_{A_1}^2}\\
			& \le \frac{|y_c|^2_{A_0}+\sum_{j=1}^J |y_j|^2_{A_0}}{|y|^2_{A_0}}
			+\dfrac{\|y_c\|^2_{A_1}+\sum_{j=1}^J \|y_j\|^2_{A_1}}{\|y\|_{A_1}^2}.
		\end{align*}
		The upper bound in \eqref{eq:cond-number-estimate} is obtained by taking infimum over all decompositions of $y \in \Ker(A_0)^\perp$ on the right hand side.  
		To show the lower bound in \eqref{eq:cond-number-estimate}, note that for any decomposition, we have
		\begin{align*}
		\| v \|_A^2 = \| v_c + \sum_{j=1}^J v_j \|_A^2 \leq 2 \| v_c \|_A^2 + 2 \| \sum_{j=1}^J v_j \|_A^2 \leq 2 \| v_c \|_A^2 + 2 \sum_{j=1}^J \| v_j \|_A^2.
		\end{align*} 
        Therefore, taking the infimum over all possible decompositions, we can obtain the lower bound and conclude the proof. 
	\end{proof}
	
	Now we have results on our preconditioning method's uniform condition number estimation. In the following two subsections, we show that the assumptions  \eqref{eq:kernel-condition} and \eqref{eq:stable-condition} are valid in the context of the algebraic systems that we are considering.

	
	\subsection{Kernel decomposition condition} \label{sec:kernel-condition}
	
	We continue with defining the subspace splitting that will satisfy the kernel condition in \eqref{eq:kernel-condition}. At the same time, we bear in mind to choose a decomposition that intuitively follows the \eqref{eq:stable-condition} as well.
	
	Consider characterizing the kernel of the matrix $ A_0 $ as
	\begin{align} \label{eq:kernel-writeup}
		\Ker(A_0) 
		& = \{ v = (v_\Omega, v_\Gamma) \in V \, : \, a_0(v, v) = m_\Gamma(Rv, Rv) = 0 \} \nonumber \\
		& = \{ v = (v_\Omega, v_\Gamma) \in V \, : \, v_\Gamma = \sigma(v_\Omega)  \} \nonumber \\
		& = \left\{ \begin{pmatrix} I_\Omega \\ \sigma \end{pmatrix} v_\Omega  \, : \, v_\Omega \in V_\Omega \right\}.
	\end{align}
	Therefore, we can fully represent $ \Ker(A_0) $ with the vectors from the subspace $ V_\Omega $. 
	We take this into account when constructing the subspaces $ V_j \subset V, \, V_c+\sum\limits_{j = 1}^J V_j = V $. 
	Descriptively, we should find a partition of $ V $ so that each part contains at least one spanning vector of $ \Ker(A_0) $ with a minimal overlap between the subspaces. Note that we first start with the subspaces $ V_j, \, j \geq 1 $ that define the Schwarz preconditioner $ S $ and then we construct the coarse space $ V_c $ via vertex aggregation since we consider aggregation-based AMG.
	
	For each graph vertex $ j \in \mathcal{V}_\Omega $, define the neighborhood of $ j $ in terms of the sparsity pattern of the operator $ R $, that is
	\begin{equation}\label{eq:neighborhoods}
		\mathcal{N}_j = \{ i \in \mathcal{V}_\Gamma \, : \, (\sigma(e_j^\Omega))_i \neq 0, \; e_j^{\Omega} \in V_\Omega \}, \quad 
		\text{where} \ (e_j^\Omega)_k = \begin{cases}
			1, \quad k = j \\
			0, \quad k \neq j
		\end{cases}
		\quad j \in \mathcal{V}_\Omega.
	\end{equation}
	Specifically, the neighborhoods $ \mathcal{N}_j $ are the subsets of all the vertices in $ \mathcal{V}_\Gamma $ that the vertex $ j \in \mathcal{V}_\Omega $ restricts to in terms of $ \sigma $. Note that $ \mathcal{N}_j = \emptyset $ if $ \sigma(e_j^\Omega) = 0 $, which means that that vertex $ j \in \mathcal{V}_\Omega $ does not connect to any vertex in $ \mathcal{V}_\Gamma $ by the action of the operator $ \sigma $.
	Since $ \sigma $ is surjective, then we have that 
 \begin{equation}\label{eq:index-split} 
 \bigcup_{j = 1}^{N_\Omega} \mathcal{N}_j = \mathcal{V}_\Gamma\quad\mbox{and}\quad \bigcup_{j = 1}^{N_\Omega} (\mathcal{N}_j \cup \{ j \}) = \mathcal{V}_\Gamma \cup \mathcal{V}_\Omega = \mathcal{V}.
 \end{equation} 
 Hence, we have constructed a partition of the vertices of the graph in overlapping subsets and we use this below to define the partition of unity needed in the analysis of a Schwarz smoother.

 We now consider the subspaces $V_j\subset V$ defined as follows 
	\begin{equation}\label{eq:space-decomposition}
		V_j = \Span \left( \left\{ \begin{pmatrix} e^\Omega_j \\ 0 \end{pmatrix} \right\} \cup \left\{ \begin{pmatrix} 0 \\ e^\Gamma_i \end{pmatrix}, \, i \in  \mathcal{N}_j \right\} \right) \subset V,  \qquad j \in \mathcal{V}_\Omega.
	\end{equation}
	Since $ \bigcup_{j = 1}^{N_\Omega} (\mathcal{N}_j \cup \{ j \}) = \mathcal{V} $, it is straightforward to see that $ \sum\limits_{j = 1}^J V_j = V $. It also follows that the number of subspaces is $ J \leq N_\Omega + 1 $ and if any $ \mathcal{N}_j = \emptyset $ then $ V_j = \Span \left\{ \begin{pmatrix} e_j^\Omega \\ 0 \end{pmatrix} \right\} $. More precisely, that means that outside the domain of influence of the restriction operator $ \sigma $, the subspaces are defined only on the support of the local finite element function of that degree of freedom. In turn, that means that "around" $\Gamma$ we will have an overlapping Schwarz method as the smoother, while in the rest of the domain, the subspaces define a standard pointwise smoother (Jacobi or Gauss-Seidel method).
	
	Next, we define the coarse space $ V_c $ given by the UA-AMG method. Other constructions of coarse spaces are also possible, but we choose UA-AMG because the analysis in this case is more transparent and concise. The aggregates are constructed from the set of vertices $ \mathcal{V} = \{ 1, 2, \dots, N \} $ as follows:
	\begin{subequations} \label{eq:aggregation}
		\begin{align}
			&\text{ Splitting: } \{1,\ldots,N\}
			= \bigcup_{k = 1}^{n_{agg}} \mathfrak{a}_k, \quad \mathfrak{a}_l \cap \mathfrak{a}_k = \emptyset, \; \text{ when } l \neq k, \quad | \mathfrak{a}_k | \le C_{agg}, \quad k = 1, \dots, n_{agg}, \\
			&\text{ Approximation: for } v \in V,  \,
			v \approx v_c = \sum_{k = 1}^{n_{agg}} v_{\mathfrak{a}_k}, \, \text{ where }
			v_{\mathfrak{a}_k} = \frac{\langle \mathbbm{1}_{\mathfrak{a}_k}, v \rangle_{\ell^2} }{ |\mathfrak{a}_k| } \mathbbm{1}_{\mathfrak{a}_k}.
		\end{align}    
	\end{subequations}
	where $ \mathbbm{1}_{\mathfrak{a}_k} \in \mathbb{R}^N $ is the indicator vector on every $\mathfrak{a}_k $, $ |\mathfrak{a}_k| $ is the size of each aggregate, $ n_{agg} $ is the total number of aggregates (number of coarse grid degrees of freedom) and $ C_{agg} $ is the maximal number of fine grid vertices in any aggregate.  
 
 Associated with the splitting of the vertices given in~\eqref{eq:index-split} we now introduce a partition of unity. Consider the following matrices, each associated with the support of the vector from the frame of $ \Ker(A_0) $,
	\begin{equation}
		\chi_j = D_{\Omega}^{-1} \operatorname{diag} (\mathbbm{1}_{\mathcal{N}_j \cup \{ j \}}) \; \text { where } \; D_{\Omega} = \sum_{j = 1}^J \operatorname{diag} (\mathbbm{1}_{\mathcal{N}_j \cup \{ j \} }), \quad j = 1, \dots, J,
	\end{equation}
	where $ \mathbbm{1}_{\mathcal{N}_j \cup \{ j \}} $ being the indicator vectors on a subset of vertices $ \mathcal{N}_j \cup \{ j \} \subset \mathcal{V} $. Clearly, the matrices $ \chi_j \in \mathbb{R}^{N \times N}$ and $ \sum_{j = 1}^J \chi_j = I$. The latter identity just means that $\{\chi_j\}_{j=1}^J$ form a partition of unity. 
 In addition, we have that $ \chi_j\chi_k=D_{\Omega} \chi_j(\chi_j\chi_k) $ and $ D_{\Omega} \chi_j(\chi_j-\chi_k)\chi_k	$ is positive semidefinite. 
	
	Finally, the full subspace decomposition is given as following: for $ v \in V $, 
	\begin{equation} \label{eq:decomposition}
		v = v_c+\sum_{j = 1}^J v_j, \ \text{where} \ v_c = \sum\limits_{k = 1}^{n_{agg}} v_{\mathfrak{a}_k} \ \text{and} \  v_j = \chi_j(v-v_c), \ \ j = 1, \cdots, J.
	\end{equation}
	
	Based on the definitions of the kernel $ \Ker(A_0) $ and the subspaces, we can verify that the space decomposition \eqref{eq:space-decomposition} satisfies the kernel decomposition condition \eqref{eq:kernel-condition}, which is summarized in the following proposition.
	\begin{proposition} \label{prop:kernel-proof}
		Let $ V_j, \, j = c, 1, 2, \dots, J $ be the subspaces of $ V $ defined in \eqref{eq:space-decomposition} and \eqref{eq:aggregation}. Then $\{ V_c \} \cup \{ V_j \}_{j=0}^J $ satisfy the kernel decomposition condition \eqref{eq:kernel-condition}.
	\end{proposition}
	\begin{proof}
		By definition, we have that $ \Ker(A_0) \cap V_c \subseteq \Ker(A_0)$ and $ \sum\limits_{j = 1}^{J} \Ker(A_0) \cap V_j \subseteq \Ker(A_0) $. On the other hand, for any $ v \in \Ker(A_0) $ we know that $ v = \begin{pmatrix} I_\Omega \\ \sigma \end{pmatrix} v_\Omega $ for some $ v_\Omega \in V_\Omega $.  Equivalently, the column vectors $ \{ \mathbbm{1}_{\mathcal{N}_j \cup \{ j \}} \}_{j = 1}^{N_\Omega} $ span $ \Ker(A_0) $, which can be expanded as follows
		\begin{equation}\label{key}
			v = \sum\limits_{j = 1}^{N_\Omega} (v_\Omega)_j \underbrace{ \mathbbm{1}_{\mathcal{N}_j \cup \{ j \} } }_{\in \Ker(A_0)}
			= \sum\limits_{j = 1}^{N_\Omega} \underbrace{ (v_\Omega)_j 
				\underbrace{ \left[ 
				 \begin{pmatrix} e^\Omega_j \\ 0 \end{pmatrix} + \sum\limits_{i \in \mathcal{N}_j} \begin{pmatrix} 0 \\ e^\Gamma_i \end{pmatrix} 
				\right]
				}_{\in V_j}
			}_{\in \Ker(A_0) \cap V_j}
			\in \sum\limits_{j = 1}^{N_\Omega} \Ker(A_0) \cap V_j
			\subseteq \sum\limits_{j = 1}^{J} \Ker(A_0) \cap V_j.
		\end{equation}
	\end{proof}


	\subsection{Stable decomposition condition} \label{sec:stable-condition}
	Now that we have shown the kernel decomposition condition, we show that the same subspace decomposition is also stable in both $ A_0 $- and $ A_1 $-inner products. 
	
	\subsubsection{Stability condition in $ A_1 $} \label{sec:stability-a1} 
	We first focus on $ A_1 $ and the coarse space estimates, and present the following immediate result on the stability estimates of the coarse space $ V_c $.
	
	\begin{lemma} \label{lemma:stability-approximation} 
		For $ v \in V $ and its coarse grid approximation $ v_c = \sum\limits_{k = 1}^{n_{agg}} v_{\mathfrak{a}_k} $ we have
		\begin{equation}\label{e:s-a}
			\| v - v_c \|_{D_1}^2 \lesssim \| v \|_{A_1}^2 \;
			\text{ and } \; \| v_c \|_{A_1}^2 \lesssim \|v\|_{A_1}^2,
		\end{equation}
		with $ D_1 $ is the diagonal of $ A_1 $.
	\end{lemma}
	\begin{proof}
		The first estimate follows from Poincar\'{e} inequality \eqref{eq:graph-poincare} on each aggregate
		$ \mathfrak{a}_k $, $ k = 1, \dots, n_{agg} $, i.e.,
		\begin{equation} \label{eq:proof-lemma-stab-a}
			\|v - v_c \|_{D_1}^2 = \sum\limits_{k = 1}^{n_{agg}} \| v - v_{\mathfrak{a}_k} \|_{D_1, \mathfrak{a}_k}^2 
			\lesssim \sum_{\mathfrak{a}_k} c_{\mathfrak{a}_k} \| v \|_{A_1, \mathfrak{a}_k}^2 
			\lesssim \left(\max_{\mathfrak{a}_k} c_{\mathfrak{a}_k} \right) \| v \|_{A_1}^2,
		\end{equation}
		where $ c_{\mathfrak{a}_k} $ are the Poincar\'{e} constants on each aggregate $ \mathfrak{a}_k $ cf.~\cite{1989JerrumM_SinclairA-ab}, \cite{2017XuZikatanov-aa}. In addition, $\| \cdot \|_{D_1, \mathfrak{a}_k}$ and $\| \cdot \|_{A_1, \mathfrak{a}_k}$ denote the $D_1$- and $A_1$-norms restricted to the aggregate $\mathfrak{a}_k$, respectively. 
		
		 \noindent The second estimate follows from the triangle inequality, \eqref{eq:A-le-D}, and \eqref{eq:proof-lemma-stab-a}, namely,
		\begin{equation*}
			\| v_c \|_{A_1}^2 
			\lesssim \| v \|_{A_1}^2 + \| v - v_c \|_{A_1}^2
			\lesssim  \| v \|_{A_1}^2 + \| v - v_c \|_{D_1}^2 
			\lesssim \| v \|_{A_1}^2.
		\end{equation*}    
	\end{proof}

	Next we consider estimates on the specific subspaces used for the Schwarz smoother \eqref{eq:preconditioner-b} (with using $A_1$ instead of $A$), which is defined by taking the infimum over all possible decompositions of any fine scale function $v \in V$ based on the space decomposition $ V = V_c+\sum\limits_{j = 1}^{J} V_j$.  We want to show the stability condition in $ A_1 $-norm, that is, choosing the decomposition \eqref{eq:decomposition} we achieve a uniform bound on that infimum. 
	\begin{proposition} \label{prop:stable-A1}
		For any $ v \in V $, let $ \{ v_c \} \cup \{ v_j \}_{j = 1}^J $ be the subspace decomposition defined in \eqref{eq:decomposition}. Then,
		\begin{equation} \label{eq:stability-A1}
			 \| v_c \|^2_{A_1} + \sum_{j = 1}^J \| v_j \|^2_{A_1} \lesssim \| v \|^2_{A_1}.
		\end{equation}
  The stability constants hidden in~"$\lesssim$" depend on the maximal number of nonzeroes per row in $ A_1 $, the maximum of the local Poincar\'{e} constants of each aggregate $ \mathfrak{a}_k $, $ k = 1, \dots, n_{agg} $, and the maximum over the number of intersections between the subspaces $ V_j $, $ j = 1, \dots, J $.
	\end{proposition}

	\begin{proof}
		Using \eqref{eq:A-le-D}, the definition of $\chi_j$, and \Cref{lemma:stability-approximation}
  we obtain that
		\begin{align*}
			\| v_c \|^2_{A_1} + \sum_{j = 1}^J \| v_j \|^2_{A_1} & = \| v_c \|^2_{A_1} + \sum_{j = 1}^J \| \chi_j(v - v_c) \|^2_{A_1} \\
			& \lesssim \| v_c \|^2_{A_1} + \sum_{j = 1}^J \| \chi_j(v - v_c) \|^2_{D_1}~~~~~~~~~~~~~~~~\mbox{(from \eqref{eq:A-le-D})}\\
			& \lesssim \| v_c \|^2_{A_1} + \| v - v_c \|^2_{D_1}~~~~~~~~~~~~~~~~~~~~~~~~~~\mbox{(by the definition of $\chi_j$)}\\
			& \lesssim \| v \|^2_{A_1}~~~~~~~~~~~~~~~~~~~~~~~~~~~~~~~~~~~~~~~~~~~~~~\mbox{(from \Cref{lemma:stability-approximation})}.
		\end{align*}
		%
	\end{proof}

	\subsubsection{Stability condition in $ A_0 $} \label{sec:stability-a0}
	Finally, to prove the stability of the decomposition in the $ A_0 $-inner product, it is necessary to specify some properties of the lower-order term $ a_0(\cdot, \cdot) $. While it slightly limits the applicability of our approach, the problems we are considering in \Cref{sec:implementation} adhere to the required assumptions.

	Let $ L : \mathcal{V}_\Gamma \to \mathcal{V}_\Omega $ be a metric function on the vertices of the graph such that
	\begin{equation} \label{eq:def-metric-graph}
		L(i) = \arg\min \{ \dist(i, j), \; j \in \mathcal{V}_\Omega \}, \quad i \in \mathcal{V}_\Gamma.
	\end{equation}
	The metric $ \dist(\cdot, \cdot) $ can be any metric between the graph vertices in $ \mathcal{V}_\Gamma $ and $ \mathcal{V}_\Omega $. For example, in reduced EMI example, it can be $\dist(i,j) := \| p_i-p_j \| $ where $p_i$ and $p_j$ are the spatial locations of the vertex $i \in \mathcal{V}_{\Gamma}$ and $j \in \mathcal{V}_{\Omega}$, respectively. 
	The function $ L(\cdot) $ is single-valued, but its pseudoinverse is possibly set-valued and can be extended to the whole $ \mathcal{V} $. More specifically, the inverse $ L^{-1} : \mathcal{V}_\Omega \to \mathcal{V}_\Gamma $ and its extension $ \tilde{L} : \mathcal{V}_\Omega \to \mathcal{V} $ are defined as
	\begin{equation}\label{eq:inverse-metric}
		L^{-1}(j) = \{ i \in \mathcal{V}_\Gamma \, : \, L(i) = j \}, \qquad \tilde{L}(j) = \{ j \} \cup L^{-1}(j), \quad j \in \mathcal{V}_\Omega.
	\end{equation}
	Note that if for $ j_1, j_2 \in \mathcal{V}_\Omega $, $ j_1 \neq j_2 $, then $ L^{-1}(j_1) \cap L^{-1}(j_2) = \emptyset $ and consequently $ \tilde{L}(j_1) \cap \tilde{L}(j_2) = \emptyset $. Also, it is possible to have $ L^{-1}(j) = \emptyset $ and that holds for $ j \in \mathcal{V}_\Omega $ that do not interpolate any $ i \in \mathcal{V}_\Gamma $. For example, in the reduced EMI equations, that applies to the "interior" DOFs of the 3D subdomain that have no contribution in the averaging operator $ \sigma $. On the other hand, note that $ \tilde{L}(j) \neq \emptyset $ and is surjective for all $ j \in \mathcal{V} $.
	This motivates us to redefine the lower-order term $ a_0(\cdot, \cdot) $ from \eqref{eq:graph-bilinear-c} to
	\begin{equation}\label{eq:redefine-a0}
		a_0(u, v) = \sum_{i \in \mathcal{V}_\Gamma}  m_i (u_{\Gamma, i} - u_{L(i)}) (v_{\Gamma, i} - v_{L(i)}),
	\end{equation}
	and the seminorm becomes
	\begin{equation}\label{eq:new-A0-norm}
		| v  |_{A_0}^2 = \sum_{j \in \mathcal{V}_\Omega} \sum_{i \in L^{-1}(j)} m_i (v_{\Gamma, i} - v_{\Omega, j})^2, \qquad v = (v_\Omega, v_\Gamma) \in V.
	\end{equation}
	In this setting, we can easily represent the kernel of the matrix $ A_0 $, that is
	\begin{equation}\label{eq:new-kernel-A0}
		\Ker(A_0) = \Span \{ \mathbbm{1}_{\tilde{L}(j)}, \, j \in \mathcal{V}_\Omega \},
	\end{equation}
	thus size of the $ \Ker(A_0) $ can be as large as the number of vertices in the subdomain $ \Omega $. Notice, however, that $ \Ker(A_0) $ is not equal to $ V_\Omega $ as it also involves the coupling between elements of $ V_\Gamma $ and $ V_\Omega $.
	By definition, the aggregates $ \mathfrak{a}_k $, $ k = 1, \dots, n_{agg} $, are disjoint subsets of $ \mathcal{V} $ and, hence, for every $ i \in \mathcal{V}$, there exists a unique $ k \in \{1, \dots, n_{agg} \}$, such that $ i \in\mathfrak{a}_k $ and we denote $ \mathfrak{a}(i):= \mathfrak{a}_k$. In accordance with this, we make an assumption that the aggregates are constructed as following:
	\begin{equation}\label{aggregation}
		\mathfrak{a}(j) = \tilde{L}(j), \quad j \in \mathcal{V}_\Omega \quad \text{ such that } \quad L^{-1}(j) \neq \emptyset. 
	\end{equation}
	Such an assumption is not restricting our approach as the sets $ L^{-1}(j) $ are well-defined and, for $ j_1 \neq j_2 $ where $ j_1, j_2 \in V_\Omega $, these sets are disjoint. The aggregation can be constructed by first choosing $ \tilde{L}(j) $ as aggregates as long as $L^{-1}(j) \neq \emptyset$.
	In such case, for a given $ v \in V $ and $ v_c = \sum\limits_{k = 1}^{n_{agg}} v_{\mathfrak{a}_k} \in V_c $ we have that $ a_0(v_c, v_c) = 0 $, that is the form $ a_0(\cdot,\cdot) $ vanishes on the coarse space. As a consequence, we have the stability of the decomposition \eqref{eq:decomposition} in $ A_0 $.
	
	\begin{proposition} \label{prop:stable-A0}
		For any $ v \in V $, let $ \{ v_c \} \cup \{ v_j \}_{j = 1}^J $ be the subspace decomposition defined in \eqref{eq:decomposition}. Then,
		\begin{equation} \label{eq:stability-A0}
			| v_c |^2_{A_0} + \sum_{j = 1}^J | v_j |^2_{A_0} \lesssim | v |^2_{A_0},
		\end{equation}
		that is, the decomposition is stable in $ A_0 $-seminorm. The stability constants depends on the maximal number of nonzeroes per row in $ A_0 $, maximum of Poincar\'{e} constants of each aggregate $ \mathfrak{a}_k $, $ k = 1, \dots, n_{agg} $ and the maximum the number of of nontrivial intersections between the subspaces $ V_j $, $ j = 1, \dots, J $.
	\end{proposition}
	
	\begin{proof}
		The proof follows analogously to the proof of \Cref{prop:stable-A1} by replacing $ A_1 $ with $ A_0 $, $ D_1 $ with the diagonal of $A_0$, and noting that $ | v_c |_{A_0} = 0 $ for the coarse space functions $ v_c = \sum\limits_{k = 1}^{n_{agg}} v_{\mathfrak{a}_k} \in V_c $.
	\end{proof}
 
	
    \subsection{On the multiplicative version of the preconditioner} \label{sec:multiplicative-smoother}
    As is well known~\cite{Xu1992SIAMReview, 1995GriebelM_OswaldP-aa}, the additive and the multiplicative versions of a two-level AMG preconditioner, including the preconditioner described earlier, are closely related. The analogue of Theorem~\ref{thm:condition-number} is basically equivalent to~\cite[Theorem~4.2]{M3AS_singular}. Then the multiplicative smoother $S_{mult}$ is clearly equivalent to the additive smoother $S$ in \eqref{eq:preconditioner} and the following inequality holds:
	\begin{equation}\label{comp}
		\frac14\| w \|_{S^{-1}}^2 \lesssim \| w \|_{S_{mult}^{-1}}^2 \lesssim \| w \|_{S^{-1}}^2.
	\end{equation}
	The upper bound only depends on the maximal degree in the graph of subspaces $\{ V_j \}_{j=1}^J$ with vertices $\{1, \ldots, J \}$ and edges given by pairs of indices $ (i,j)$  for which $ V_j \cap V_i \neq \{ 0 \}$. That is the maximal number of intersections of any subspace $ V_j $ with other subspaces $ V_i, \,i \neq j $. Such results are found in many references, see e.g. \cite{Xu1992SIAMReview}, \cite{1995GriebelM_OswaldP-aa},  \cite{2008VassilevskiP-aa}, \cite[Lemma~3.3]{2008ZikatanovLT-aa}.

    Using the techniques which we have employed in showing stability for the subspace decomposition in $A_1$-norm, we can show the "weak" approximation property, which is a necessary and sufficient condition for uniform two-level AMG convergence, \cite[Theorem~3.5]{2008ZikatanovLT-aa}, \cite{hu_vassilevski}. The approximation result reads: For any $v\in V$ there exists $v_c\in V_c$ such that the following estimates hold independently of parameters $ h $ and $ \gamma $:
    \begin{equation}\label{meq:wap}
    	\|v - v_c \|_{S_{mult}^{-1}}^2 \lesssim \| v \|_{A_\iota}^2, \quad \iota \in\{0,1\},
    \end{equation}
    where, when $\iota=0$ we can take $v\in \operatorname{Ker}(A_0)^{\perp_{A_1}}$. This estimate, which follows from the results we have shown for the additive preconditioner, gives the uniform convergence of the multiplicative method.

  
	\section{Implementation} \label{sec:implementation}
	
	We dedicate this section to explaining what the convergence conditions and how they can be utilized to construct uniformly convergent multilevel method in different applications, namely with regards to example problems in \Cref{sec:examples}. Moreover, we confirm the theory with numerical results that are obtained using software components HAZniCS \cite{haznics}.
          Unless stated otherwise the finite element problems are assembled using FEniCS
          \cite{logg2012automated} and FEniCS\textsubscript{ii} \cite{kuchta2020assembly}.
	
	
	\subsection{Bidomain model} \label{sec:impl-bidomain}
	
	Consider $ A $ and $ M $ to be the matrix representations of the Laplacian $ -\Delta : V \to V' $ and the $ L^2 $-inner product on $ V $, respectively. With $ \bar{V} = V \times V $, let $ \bar{K} : \bar{V} \to \bar{V}'$ represent the system operator in the equation \eqref{eq:bidomain-problem}.
	Furthermore, denote a coarse space $ V_c \subset V $ and the corresponding (surjective) prolongation operator $ P : V_c \to V $ and combine $ \bar{V}_c = V_c \times V_c $ and $ \bar{P} : \bar{V}_c \to \bar{V} $.
	As mentioned in \Cref{sec:multiplicative-smoother}, the necessary and sufficient condition for convergence of the two-level AMG method for solving \eqref{eq:bidomain-problem} is the weak approximation property. If $ \bar{S} : \bar{V}' \to \bar{V} $ is the smoother for the two-level AMG method on $ \bar{K} $, then we want that for any $ y \in \bar{V} $ there exists $ y_0 \in \bar{V}_c $ such that
	\begin{equation}\label{eq:bidomain-weak-property}
		\| y - \bar{P} y_c \|^2_{\bar{S}^{-1}} \leq C \| y \|^2_{\bar{K}},
	\end{equation}
	for some $ C > 0 $. This is satisfied for standard (point-wise) smoothers such as Jacobi or Gauss-Seidel method, but they do not guarantee that the bound is independent of $ \gamma $. On the other hand, following the theory derived in \Cref{sec:two-level-AMG}, we can show that the following smoother
	\begin{equation} \label{eq:bidomain-smoother}
		\bar{S}^{-1} = \begin{pmatrix}
			\alpha_e D_A + \gamma D_M & - \gamma D_M \\
			- \gamma D_M & \alpha_i D_A + \gamma D_M
		\end{pmatrix},
	\end{equation}
	satisfies the stability and kernel conditions, with $ D_M $ and $ D_A $ being diagonals of matrices $M$ and $A$, respectively. Actually, we can directly prove the weak approximation property in this case, by relying on two results from \eqref{eq:A-le-D}: For $ v \in V $ there exists $ v_c \in V_c $ such that
	\begin{equation} \label{eq:bidomain-weak-jacobi}
			\| v - P v_c \|^2_{D_A} \leq C_A \| v \|^2_{A} \quad \text { and } \quad \| v - P v_c \|^2_{D_M} \leq C_M \| v \|^2_{M},
	\end{equation}
	with $ C_A, C_M > 0 $ depending only on the number of nonzeroes per row in $ A $ and $ M $, respectively.
	Take $ y = \begin{pmatrix} v_e \\ v_i \end{pmatrix} \in \bar{V} $. Define $ w^{+} = \half (v_e + v_i) $ and $ w^{-} = \half (v_e - v_i) $, so that 
	\begin{equation} \label{eq:function-decomp-bidomain}
		y = \begin{pmatrix}
			w^{-} \\ - w^{-}
		\end{pmatrix} 
		+
		\underbrace{
			\begin{pmatrix}
				w^{+} \\ w^{+}
			\end{pmatrix}
		}_{\in \Ker(\bar{M})}.
	\end{equation}
	We know that $ w^{+}, w^{-} \in V $ so there exist $ w^{+}_c, w^{-}_c \in V_c $ that satisfy \eqref{eq:bidomain-weak-jacobi}. Taking $ y_c \in \bar{V}_c $ as 
	$ y_c = \begin{pmatrix}
		w^{-}_c \\ - w^{-}_c
	\end{pmatrix} 
	+
	\begin{pmatrix}
		w^{+}_c \\ w^{+}_c
	\end{pmatrix}, $
	it follows that
	\begin{align*}
		\| y - \bar{P} y_c \|^2_{\bar{S}^{-1}} 
		& = \| \begin{pmatrix} w^{-} - P w^{-}_c \\ -(w^{-} - P w^{-}_c) \end{pmatrix} \|^2_{\bar{S}^{-1}} + \| \begin{pmatrix} w^{+} - P w^{+}_c \\ w^{+} - P w^{+}_c \end{pmatrix} \|^2_{\bar{S}^{-1}} \\
		& = (\alpha_e + \alpha_i) \left( \| w^{-} - P w^{-}_c \|^2_{D_A} + \| w^{+} - P w^{+}_c \|^2_{D_A} \right) + 2 \gamma \| w^{-} - P w^{-}_c \|^2_{D_M} \\
		& \leq \max\{C_A, C_M \} \left( (\alpha_e + \alpha_i) \left( \| w^{-}\|^2_{A} + \| w^{+} \|^2_{A} \right) + 2 \gamma \| w^{-} \|^2_{M} \right) \\
		& = \max\{C_A, C_M \} \| y \|^2_{\bar{K}}.
	\end{align*}
	It is possible to notice where a pointwise smoother would fail to control (with regards to $\gamma$) functions $ y \in V $ where the kernel part in \eqref{eq:function-decomp-bidomain} is nonzero. For example, for a Jacobi smoother given as the diagonal of the system matrix $\bar{K} $ (which only contains $D_A$), we are left with an extra term $ 2\gamma\| w^{-} \|_M^2 $ in the third line in the above proof which is unbounded in $ \gamma $.
	
	Interestingly, UA-AMG method for $ \bar{K} $ handles the near-kernel functions naturally. Due to the two-by-two block structure of $ \bar{M} $ and the fact that $ M $ is a mass matrix, we can obtain a special prolongation operator
	\begin{equation} \label{eq:prolongation}
		\bar{P} = \begin{pmatrix}
			I_V \\ I_V
		\end{pmatrix} : \bar{V}_c \to \bar{V} \text{ and } \bar{V}_c = V.
	\end{equation}
	That is, in the two-level method, each coarse space DOF is constructed by combining two fine scale DOFs that are coupled with $ \bar{M} $ and the total number of DOFs in the coarse space is half of the fine scale space. Note that $ \Range(\bar{P}) $ is exactly the near kernel of $ \bar{K} $ which implies that the coarse grid correction will
	handle this type of functions and simultaneously preserve stability. 
 
    This kind of construction is possible since the coupling is present in the whole domain and the DOFs in each subdomain align with each other. The more interesting case of 
    lower-dimensional interface coupling is given in the following subsections, where both a special prolongation and a Schwarz smoother are needed for uniform convergence. Even though, the aim of this example is to show in simple context what the conditions derived in \Cref{sec:convergence-AMG} mean and how uniform convergence can be obtained in any multilevel setting.
	
	This is also numerically confirmed in the following results. First, we consider the problem \eqref{eq:bidomain-problem} on a shape regular triangulation of a unit square domain $ \Omega = (0, 1)^2 $ and discretized with the continuous linear finite elements $ \mathbb{P}_1 $. To solve the problem, we use a Conjugate Gradient (CG) method preconditioned with different configurations of the AMG method and we test the solver performance against mesh refinement and coupling strength. The common settings in all configurations are a two-level aggregation-based AMG method, one W-cycle per iteration and a direct solver (UMFPACK \cite{lu_from_umfpack}) on the coarse grid. We compare the performance of the UA-AMG and SA-AMG with or without the special prolongation \eqref{eq:prolongation}, and with the kernel-aware Schwarz smoother \eqref{eq:preconditioner-b} or a pointwise (Gauss-Seidel) smoother. Both smoothers are applied in symmetric multiplicative way. The results are given in \Cref{fig:bidomain_2d}. As expected, regular AMG method is fairly stable with regards to mesh refinement but without a smoother satisfying kernel conditions \eqref{eq:kernel-condition} we see a significant increase in number of iterations and condition number. Adding a kernel-aware Schwarz smoother, we see a more robust performance with regards to the coupling parameters $ \gamma $ but the smoother may influence the stability conditions \eqref{eq:stability-A0} resulting (in case of UA-AMG) in the increase of number of iterations for finer meshes. This behavior is stabilized using the prolongation operator \eqref{eq:prolongation} that confirms the theory that both kernel and stability conditions are necessary for the uniform convergence.
	
	\begin{figure}
		\centering
		\includegraphics[width=0.9\textwidth]{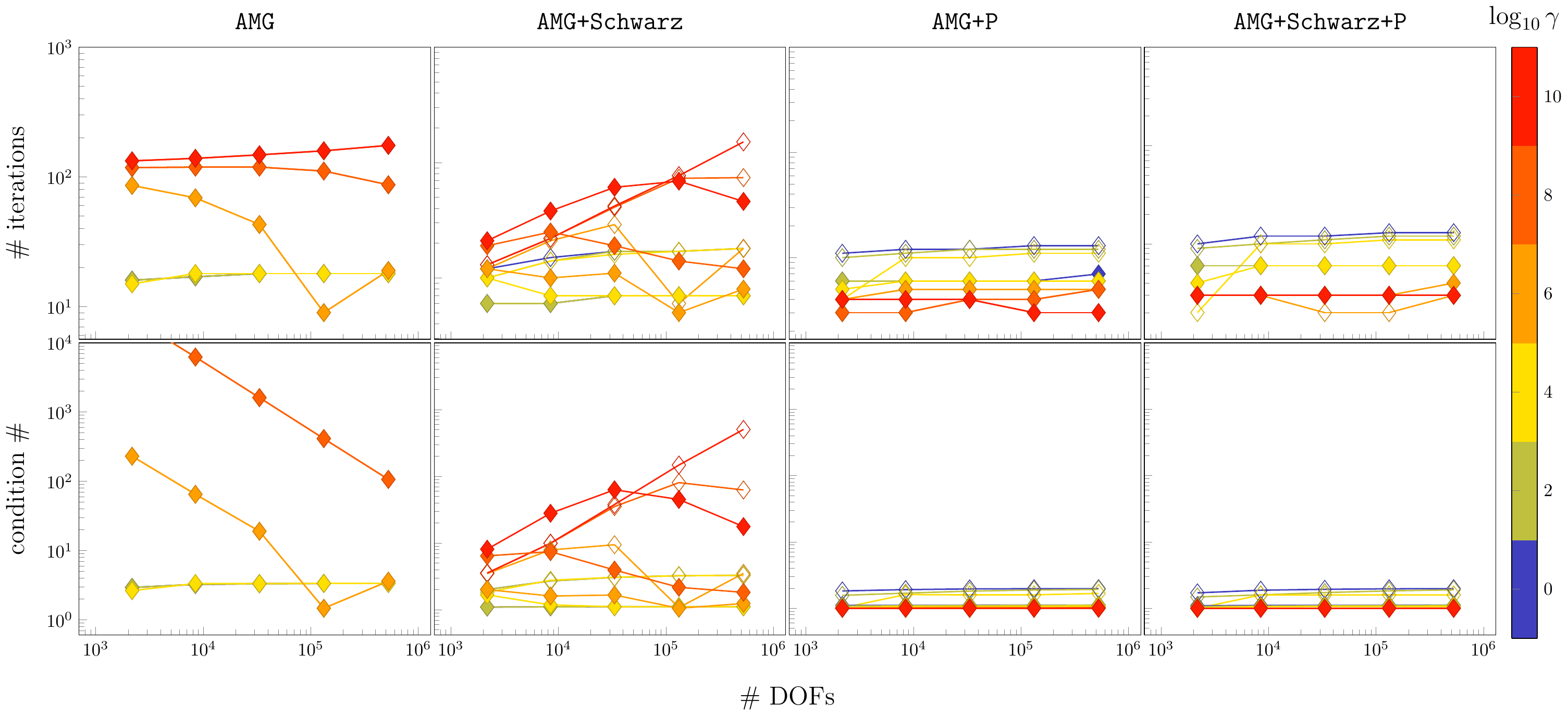}
		\vspace{-5pt}
		\caption{Number of iterations and estimated condition number of the CG method preconditioned with aggregation-based AMG to solve the bidomain problem \eqref{eq:bidomain-problem} in $ \Omega=(0,1)^2 $. We show the performance of the AMG preconditioner without regular prolongation and pointwise smoother operators, with using the Schwarz smoother \eqref{eq:preconditioner-b} satisfying the kernel conditions, with special prolongation in \eqref{eq:prolongation} and with using both the Schwarz smoother and special prolongation, which are respectively shown in first, second, third and fourth column. Hollow marks indicate using UA-AMG method and full marks using SA-AMG method. The color of the lines indicates the magnitude of the coupling parameter $\gamma$ ranging from $1$ (blue) to $10^{10}$ (red).
		}
		\label{fig:bidomain_2d}
	\end{figure}
	
	Additional performance results of the AMG-preconditioned CG method are shown in a three-dimensional setting in \Cref{fig:bidomain_3d}. We consider again a shape regular simplicial mesh of the unit cube domain $ \Omega = (0, 1)^3 $ and $ \mathbb{P}_1 $ finite element discretization. We study the number of CG iterations, condition number and CPU time consumption with regards to mesh refinement and coupling parameter magnitude. We can conclude that while the AMG preconditioner, both with and without using the prolongation operator \eqref{eq:prolongation}, performs uniformly with regards to the system size, the prolongation operator is definitely needed to achieve stable performance for larger values of $\gamma$.
	
	\begin{figure}
		\centering
		\includegraphics[width=0.8\textwidth]{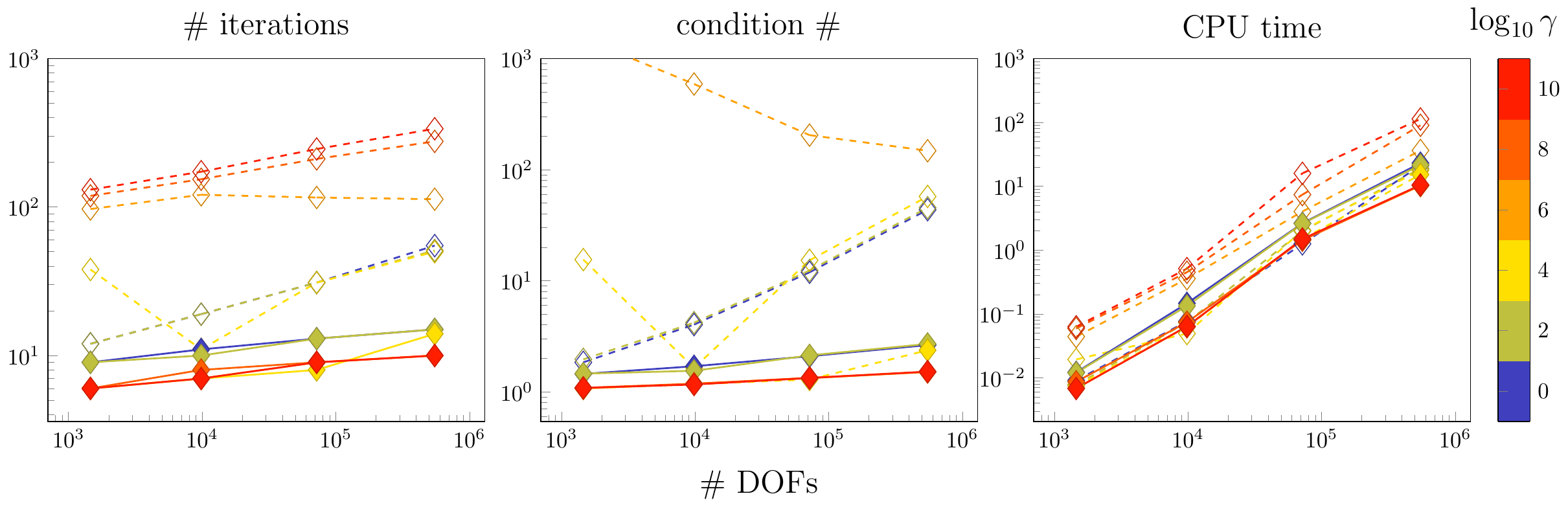}
		\vspace{-5pt}
		\caption{Number of iterations, estimated condition number and total solving CPU time of the CG method preconditioned with UA-AMG to solve the bidomain problem \eqref{eq:bidomain-problem} in $ \Omega=(0,1)^3 $, shown in first, second and third column respectively. We compare the performance of the standard AMG preconditioner (marked with dashed lines) and AMG with using the special prolongation in \eqref{eq:prolongation} (marked with full lines). The color of the lines indicates the magnitude of the coupling parameter $\gamma$ ranging from $1$ (blue) to $10^{10}$ (red).
		}
		\label{fig:bidomain_3d}
	\end{figure}

        \begin{remark}[Geometric multrigrid]
          The bidomain smoother \eqref{eq:bidomain-smoother} can be viewed as a Schwarz smoother for
          degrees of freedom located at a ``star''-patch/macroelement of each
          vertex. To show robustness of our approach in the geometric multigrid
          setting we implemented a multigrid preconditioner for \eqref{eq:bidomain-problem}
          using the PCPATCH framework \cite{pcpatch} for the 
          required space decompositions. The finite element discretization
          was done with Firedrake \cite{firedrake} library. Using the 2D geometry from
          \Cref{sec:impl-bidomain} we show in \Cref{tab:bidomain_2d_firedrake}
          the number of CG iterations required for convergence with relative error tolerance
          of $10^{-10}$. The preconditioner applied a single $V$-cycle per CG iteration.
          We observe that the iteration counts are bounded in mesh size and the coupling
          parameter $\gamma$.
  \begin{table}
    \begin{center}
      \footnotesize{
        \begin{tabular}{l|cccccc}
          \hline
          \backslashbox{\#DOFs}{$\gamma$} & 1 & $10^2$ & $10^4$ & $10^6$ & $10^8$ & $10^{10}$\\
          \hline
2178 & 19 & 19 & 18 & 19 & 19 & 20\\
8450 & 19 & 19 & 18 & 18 & 19 & 19\\
33282 & 19 & 19 & 17 & 17 & 19 & 19\\
132098 & 19 & 19 & 17 & 17 & 17 & 18\\
526338 & 19 & 19 & 17 & 17 & 17 & 17\\
          \hline
        \end{tabular}
      }
      \caption{Number of preconditioned CG iterations using geometric multigrid
        and smoother \eqref{eq:bidomain} for bidomain equation \eqref{eq:bidomain-problem}
        with $\Omega=(0, 1)^2$ and discretization by $\mathbb{P}_1$-elements.
      }
      \label{tab:bidomain_2d_firedrake}
    \end{center}
  \end{table}
        \end{remark}

	\subsection{EMI model} \label{sec:results-emi} 
        We next investigate performance of the proposed AMG method for solving
        the EMI model \eqref{eq:emi_2d} in both 2D and 3D setting. In
        the former case we let $\Omega_i=(0, 1)\times(0, \frac{1}{2})$, $\Omega_e=(0, 1)\times(\frac{1}{2}, 1)$
        while in 3D $\Omega_i=(0, 1)^2\times(0, \frac{1}{2})$, $\Omega_e=(0, 1)^2\times(\frac{1}{2}, 1)$.
        In both cases Dirichlet conditions are prescribed on boundary surfaces parallel with
        the interface $\Gamma$. Neumann conditions are set on the remaining parts of the boundary.
        The systems \eqref{eq:emi_2d_operator} are then discretized by $\mathbb{P}_1$ elements and solved
        by preconditioned CG method. Following \Cref{sec:impl-bidomain} the preconditioner
        uses UA-AMG with maximum of 10 levels and Schwarz smoother \eqref{eq:preconditioner-b}.
        With this setup the number of iterations required for reducing the initial preconditioned
        residual norm by $10^{10}$ is given in \Cref{tab:emi_2d}. In the both $2D$ and $3D$ cases
        the iterations are bounded in the coupling strength.
  \begin{table}
    \begin{center}
      \footnotesize{
        \begin{tabular}{l|cccccc||l|cccccc}
          \hline
          \multicolumn{7}{c||}{$\Omega_i\cup\Omega_e=(0, 1)^2$} & \multicolumn{7}{c}{$\Omega_i\cup\Omega_e=(0, 1)^3$}\\
          \hline
          \backslashbox{\#DOFs}{$\gamma$} & 1 & $10^2$ & $10^4$ & $10^6$ & $10^8$ & $10^{10}$ &
          \backslashbox{\#DOFs}{$\gamma$} & 1 & $10^2$ & $10^4$ & $10^6$ & $10^8$ & $10^{10}$\\
          \hline
4290 & 16 & 15 & 15 & 15 & 15 & 15              & 150 & 3 & 3 & 3 & 3 & 3 & 3\\         
16770 & 18 & 18 & 18 & 18 & 18 & 18             & 810 & 5 & 6 & 6 & 6 & 6 & 6\\         
66306 & 19 & 19 & 19 & 19 & 19 & 19             & 5202 & 8 & 9 & 9 & 9 & 9 & 9\\        
263682 & 20 & 21 & 20 & 20 & 20 & 20            & 37026 & 13 & 13 & 13 & 13 & 13 & 13\\ 
1051650 & 21 & 22 & 20 & 20 & 20 & 20           & 278850 & 17 & 17 & 16 & 16 & 16 & 16\\
          \hline
        \end{tabular}
      }
      \caption{Number of preconditioned CG iterations required for solving
        the EMI model \eqref{eq:emi_2d} with AMG using smoother
        \eqref{eq:preconditioner-b}.
      }
      \label{tab:emi_2d}
    \end{center}
  \end{table}
	
	
	\subsection{Reduced 3D-1D EMI model} \label{sec:impl-3d1d}
	In case of mixed-dimensional modeling, the interface coupling is usually supported on and in a close neighborhood of the lower-dimensional subdomain, where the higher-dimensional quantity is projected using a trace or an averaging operator $ \Pi_\rho $. Hence, the representation of kernel of the coupling term is tightly linked to the representation of the operator $ \Pi_\rho $. In the following, we show how the choice of the operator $ \Pi_\rho $ influences the choice of Schwarz subspaces and how the algebraic kernel and stability conditions induce a geometric multigrid method (GMG) to solve the 3D-1D coupled problem \eqref{eq:emi_3d_operator}. 
	\begin{figure}
		\centering
		\includegraphics[width=0.4\textwidth]{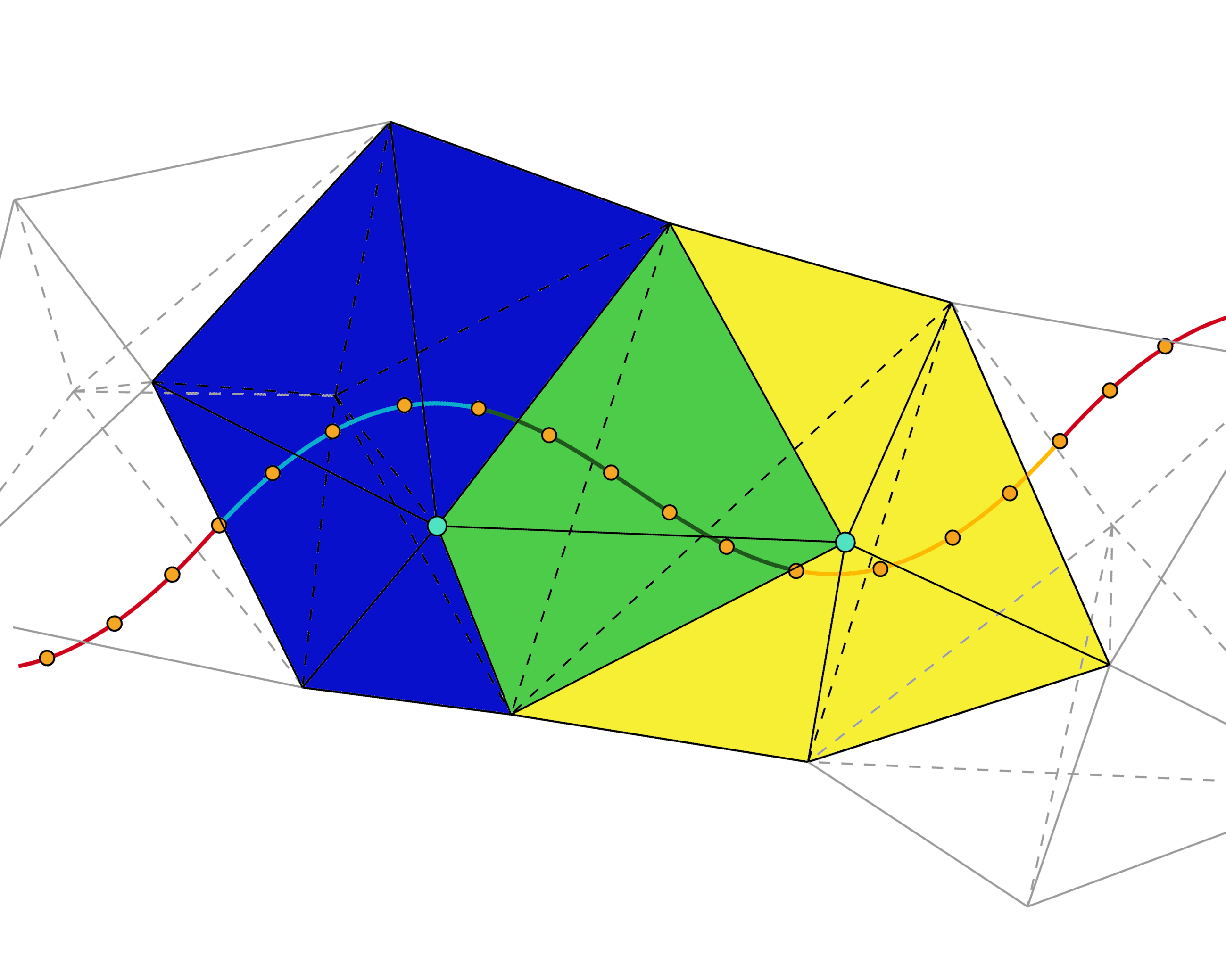}
		\hspace{20pt}
		\includegraphics[width=0.4\textwidth]{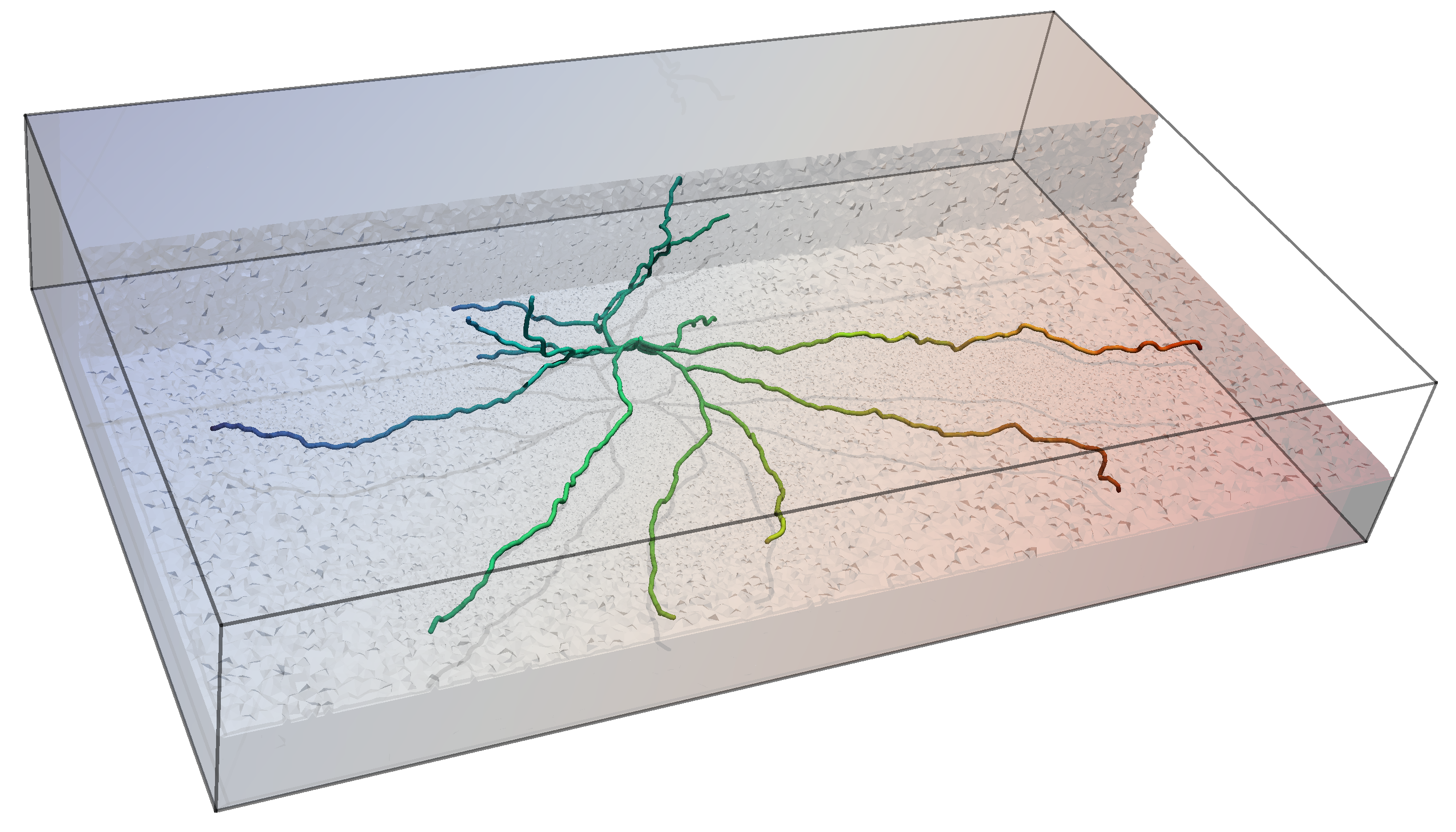}
		\vspace{-10pt}
		\caption{[Left] Illustration of overlapping Schwarz subspaces for an example of non-fitted mesh for coupled 3D-1D problem \eqref{eq:emi_3d_operator}. Assuming nodal finite element discretization, each Schwarz subspace (in blue and yellow) is local and contains the support of functions (3D and 1D) defined in DOFs that are coupled via the operator $\Pi_{\rho}$. Here, the radius of coupling $\rho$ contains only the closest 3D nodes (in light blue). The overlap is marked in green. [Right]  Domain geometry of the 3D-1D problem \eqref{eq:emi_3d_operator} The 1D domain is the neuron and the network of neuronal dendrites while the 3D domain (a shallow clip) represents extracellular space. The outline of the 3D domain is marked with black lines.
		}
        \vspace{-10pt}
		\label{fig:subspaces-3d1d}
	\end{figure}
	
	Assume we are given simplicial meshes of $ \Omega $ and $ \Gamma $, i.e. $ \mathcal{T}_h^{\Gamma} $ and $\mathcal{T}_h^{\Omega} $ that do not necessarily match. 
	Additionally assume that $ V = V_\Omega \times V_\Gamma $ is nodal-based FEM approximation, e.g. $ V_\Omega = \mathbb{P}_1 (\Omega) $ and $V_\Gamma = \mathbb{P}_1 (\Gamma)$. 
	We can define the interface (metric) operator similarly to \eqref{eq:interface_op} as $ R = \begin{pmatrix} -\Pi_\rho & I_{\Gamma} \end{pmatrix} $. Denote also $ n_\Omega = \dim V_\Omega $, $ n_{\Gamma} = \dim V_\Gamma $ and $ n = \dim V = n_\Omega + n_\Gamma $. Then, we can describe the kernel of coupling operator $ A_0 = R^T M_\Gamma R $ as
	\begin{equation}\label{eq:kernel_3d1d}
		\Ker(A_0) = \left\{  \begin{pmatrix} v_{\Omega} \\ v_{\Gamma} \end{pmatrix} \in V : \Pi_\rho v_{\Omega} = v_{\Gamma}  \right\} 
		= \left\{ \begin{pmatrix} v_{\Omega} \\ \Pi_\rho v_{\Omega} \end{pmatrix}, v_{\Omega} \in V_\Omega  \right\}.
	\end{equation}
	We can decompose the whole space as $ V = \Ker(A_0) \oplus \Ker(A_0)^{\perp} $, where $ \perp $ regards to the orthogonality in the $ A $-norm, with $ A = A_1 + \gamma A_0 $ and $ A_1 = \diag\{A_{\Omega}, A_{\Gamma} \}$. We closely follow the derivation in \Cref{sec:kernel-condition} to find a Schwarz decomposition of $ V $ that satisfies the kernel condition \eqref{eq:kernel-condition}.
	
	Assume some ordering of DOFs in $ V_\Omega $ and $ V_\Gamma $. Motivated by \eqref{eq:kernel_3d1d}, we can say that for every $ i \in \{1, \dots, n_\Omega \} $ we can define
	\begin{equation}\label{eq:neighbors}
		\mathcal{N}_\Gamma(i) = \{ k \in \{ 1, \dots, n_\Gamma \} : (R)_{ki} \neq 0 \},
	\end{equation}
	where $ (R)_{ij} $ is the element in $ R $ in $ i $-th row and $ j $-th column. 
	Therefore, for each $ i \in \{1, \dots, n_\Omega \} $ we define
	\begin{align}
		\mathcal{T}_h^i & = \{ \tau \in \mathcal{T}_h^{\Omega} : \bm{x}^i \in \tau \} \cup \{ \tau \in \mathcal{T}_h^{\Gamma} : \bm{x}^k \in \tau, \, k \in \mathcal{N}_\Gamma(i) \}, \label{triangles} \\
		\Omega_h^i & = \interior \left( \bigcup \mathcal{T}_h^i \right). \label{patches}
	\end{align}
	with $ \bm{x}^j $ are the coordinates of the node $ j $. 
	Then, the Schwarz subspaces are given by
	\begin{equation} \label{eq:schwarz_subspaces}
		V_i = \left\{ v = \begin{pmatrix} v_\Omega \\ v_\Gamma \end{pmatrix} \in V : \support(v) \subset \bar{\Omega}_h^i \right\} \quad i \in \{1, \dots, n_\Omega \} \quad \text{ and } \quad V = \sum\limits_{i = 1}^{n_\Omega} V_i.
	\end{equation}
	Using this definition and with simple computation, it follows that $ \Ker(A_0) = \sum\limits_{i = 1}^{n_\Omega} \Ker(A_0) \cap V_i $.
	
	This construction of Schwarz subspaces is used in the following numerical example. The problem is defined by the geometry
	illustrated in the right part of \Cref{fig:subspaces-3d1d}. The neuron geometry is obtained from the NeuroMorpho.Org inventory of digitally reconstructed neurons, and glia \cite{neuromorpho}. The neuron from a mouse’s brain	includes only dendrites (no axon or soma) including a total of 25 branches. It is embedded in a rectangular box of approximate dimensions 222 $\mu$m $\times$ 369 $\mu$m $\times$ 65 $\mu$m. Then, the mixed-dimensional geometry is discretized with an unstructured tetrahedron mesh fitted to $\Gamma$, i.e., line segments in the mesh of $\Gamma$ are also edges of the 3D mesh of $\Omega$. We use $\mathbb{P}_1$ finite elements for discretization of both the 3D and 1D function spaces. In total, we have 3 391 127 DOFs in 3D and 7281 DOFs for the 1D problem. Additionally, we enforce homogeneous Neumann conditions on the outer boundary of both subdomains.
	
	To obtain the numerical solution, we use the CG method preconditioned
	with the AMG method described in \Cref{sec:two-level-AMG}. The convergence is considered reached if the $l_2$ relative
	residual norm is less than $10^{-6}$. We choose the SA-AMG that uses the block Schwarz smoother (symmetric multiplicative)
	defined by the kernel decomposition \eqref{eq:schwarz_subspaces} for the DOFs that couple with regards to $\Pi_\rho$ and symmetrized Gauss-Seidel smoother on the 3D interior DOFs.
	We study the performance of our solver with regards to parameters $ \gamma $ that, resulting from the coupled membrane ODE from the full EMI model, relates to inverse of the time step size $\Delta t$, the coupling/dendrite radius $\rho$ and the membrane capacitance $C_m$ \cite{buccino2021improving}. That said, the intra- and extra-cellular conductivities and membrane capacitance parameters remain constant and fixed throughout their respective domains to $\alpha_e$ = 3 mS cm$^{-1}$, $\alpha_i$ = 7 mS cm$^{-1}$ and $C_m$ = 1 $\mu$F cm$^{-2}$ \cite{buccino2019}, while we vary the time step size and coupling radius. 
 The results given in \Cref{tab:emi_3d1d}. The first three rows use the averaging operator \eqref{eq:average_op} as the coupling operator between 3D and 1D DOFs with the radius $\rho$ as the coupling radius. Thus, the Schwarz subspaces \eqref{eq:schwarz_subspaces} are larger with larger $\rho$ and evaluating the Schwarz smoother may become expensive. On the other hand, using a weighted average and defining the value of the 1D DOF by averaging of the values of the 3D DOFs at the same element (at a distance at most $h$ from the 1D DOF)results in Schwarz subspaces of a smaller dimension. Hence, even though the number of iterations is slightly larger, the application of the Schwarz smoother in this case is computationally cheaper than the case when the averaging is done using the prescribed physical radius. Note that we still need to scale the physical parameters for intracellular space and membrane due to dimension reduction, and we use $\rho$ = 1 $\mu$m. Additionally, we note that this type of averaging corresponds to an implementation of the 3D-to-1D trace operator, but such problem formulation is not well-posed in standard norms and can cause issues with $h$-refinement \cite{gjerde2020singularity}. In summation, we observe stable number of CG iterations in all cases considered, that is, the method is robust with regards to the problem parameters.

	 \begin{table}
		\begin{center}
			\footnotesize{
				\begin{tabular}{l|cccccc}
					\hline
					\backslashbox{$\rho$ [$\mu$m]}{$(\Delta t)^{-1}$ [s$^{-1}$]} & 1 & $10^2$ & $10^4$ & $10^6$ & $10^8$ & $10^{10}$ \\
					\hline
					5.0{\color{white}*} & 2 & 2 & 2 & 3 & 3 & 4  \\             
					1.0{\color{white}*} & 2 & 2 & 2 & 3 & 3 & 4 \\         
					0.2{\color{white}*} & 2 & 2 & 2 & 3 & 4 & 4  \\           
					0.0* 				& 5 & 5 & 6 & 8 & 10 & 10  \\         
					\hline
				\end{tabular}
			}
			\caption{Number of preconditioned CG iterations required for solving
				the reduced EMI model \eqref{eq:emi_3d_operator} with AMG using smoother
				\eqref{eq:preconditioner-b}. (*) denotes that we are using the 3D-to-1D trace operator as the coupling operator.
			}
			\label{tab:emi_3d1d}
		\end{center}
	\end{table}

	\section{Conclusions} \label{sec:conclusion}
	We have developed an AMG method to solve coupled interface-driven multiphysics problems. The method is aggregation-based and introduces a custom Schwarz smoother that specifically handles the strongly weighted lower order term on the interface. We state two conditions, the kernel and the stability conditions, required for the Schwarz decomposition and the aggregation to ensure uniform convergence of the two-level method. The conditions are constructive and the method is purely algebraic only requiring information on the coupling of the interface degrees of freedom. This means that the solver can be easily implemented and applied to a variety of PDE systems. Additionally, the solver can also be realized in a geometric multigrid way, allowing for direct grid refinement around lower-dimensional inclusions. We have highlighted the effectiveness of the proposed solver to solve problems arising in modeling brain biomechanics, specifically the bidomain equations, the EMI equations and the 3D-1D coupled EMI equations.
 
	\appendix
	
	\section{Finite element matrices and graph Laplacians}   \label{sec:apx-graph-laplacians}
	We now show that a finite element discretization of an elliptic PDE is spectrally equivalent to a weighted graph Laplacian problem. The constants of the spectral equivalence depend on the polynomial degree used for the discretization.

	\begin{lemma} Let $\mathcal{T}_h$ be a simplicial mesh in $\mathbb{R}^d$ and $A_h\in \mathbb{R}^{N\times N}$ be the stiffness matrix corresponding to the discretization of an elliptic operator $L u :=-\operatorname{div}\left(\kappa \nabla u\right)$ with piece-wise polynomial space. Then $A_h$ is spectrally equivalent to a weighted graph Laplacian:
		\begin{equation}
			\langle A_h v,v \rangle_{\ell^2} \eqsim
			\langle A v,v \rangle_{\ell^2}, \quad
			\langle A v,w \rangle_{\ell^2}:=\sum_{e\in\mathcal{E}} \omega_e \delta_e v \,\delta_e w
		\end{equation}
	\end{lemma}
	\begin{proof}
		Let us consider a simplex $T\in \mathcal{T}_h$. Let $n_T$ be the number of degrees of freedom in $T$ and we define
		\begin{equation}
			\begin{aligned}
				&    | v |^2_{1,\kappa} := \int_T \kappa \nabla v\cdot\nabla v
				=
				|T| \int_{\widehat{T}}\widehat{\kappa}\left[\Phi_T^{-1} \widehat{\nabla} \widehat{v}\right]\cdot\left[\Phi_T^{-1}\widehat{\nabla} \widehat{v}\right], \quad \\
				&    | v |^2_{A} := |T|\sum_{e \in \mathcal{E}_{T}} |e|^{-2}\omega_{e}(\delta_ev)^2,\quad \omega_e > 0, \; e\in \mathcal{E}_{T}.
			\end{aligned}
		\end{equation}
		where $\mathcal{E}_{T}\subset \{1,\ldots,n_T\}\times\{1,\ldots,n_T\}$, $|e|$ is the length of the edge $e$, and
		$\omega_e$ are to be specified soon. On a shape regular mesh this can be taken to be the diameter of $T$.
		The only requirement on $\mathcal{E}_{T}$ is that these edges (pairs of degrees of freedom) contain all degrees of freedom and the corresponding graph with vertices $\{1,\ldots,n_T\}$ and edges $\mathcal{E}_{T}$  is connected. 
		The $\widehat{~}$ denotes the standard mapping to the reference simplex in $\mathbb{R}^d$:
		\begin{equation}
			\widehat{x} \in \widehat{T} \mapsto \Phi_T \widehat{x}+x_0\in T, \quad \Phi_T=\left( x_1-x_0,\ldots,x_{d}-x_0\right).
		\end{equation}
		Notice that $\left\|\Phi_T^{-1}\right\|\eqsim |e|^{-1}$, with equivalence
		constants depending on the shape-regularity of the mesh.  Then, for
		any choice of $\omega_e>0$ we have that $|v|_A$ is a norm on
		$\mathbb{R}^m/\mathbb{R}$ and similarly $|v|_{1,\kappa}$ is also a norm
		on the same finite dimensional space. These norms are equivalent
		with constants of equivalence constants depending on shape
		regularity of the mesh and the variations in $\kappa$ in the element. The
		weights $\omega_e$ can be chosen as to minimize the constants in the
		spectral equivalence. Choosing $\omega_e=\frac{1}{T}\int_T \kappa $ for
		all $e\in \mathcal{E}_T$ works in all cases when $\kappa$ is piece-wise
		smooth.  The proof is then concluded as follows (with
		$v\in \mathbb{R}^N$):
		\begin{equation}
			\langle A_h v,v\rangle_{\ell^2} =
			\sum_T|v|_{1,\kappa}^2 \eqsim 
			\sum_T|T|\sum_{e\in \mathcal{E}_T} |e|^{-2}\omega_{e}(\delta_ev)^2 \eqsim   |v|_A^2.
		\end{equation}
	\end{proof}
	
	An instructive example is to consider the piece-wise linear continuous elements
	on a shape regular mesh. Then we can choose $A$ as follows:
	\begin{equation}
		\langle Av,v\rangle = \sum_{e\in \mathcal{E}}|e|^{d-2}\omega_e(\delta_ev)^2, \quad \omega_e=\sum_{T\supset e} \frac{1}{T}\int_T \kappa.
	\end{equation}
	With such choice the constants in the lemma only depend on the shape regularity of the mesh.

	
	\bibliographystyle{siamplain}
	\bibliography{ref}

\end{document}